 \documentclass{amsart}
\usepackage{amsmath,amsthm,latexsym,amssymb,hyperref}
\usepackage{eucal}
\usepackage{mathrsfs}
\usepackage[all]{xy}
\xyoption{2cell}
\newtheorem*{thmA}{Theorem A}
\newtheorem*{thmB}{Theorem B}
\numberwithin{equation}{section}
\newtheorem{thm}{Theorem}[section]
\newtheorem{lem}[thm]{Lemma}
\newtheorem{prop}[thm]{Proposition}
\newtheorem{cor}[thm]{Corollary}

\theoremstyle{definition}
\newtheorem{defn}[thm]{Definition}
\theoremstyle{remark}
\newtheorem{rmk}[thm]{Remark}
\newtheorem{rmks}[thm]{Remarks}
\newtheorem{ex}[thm]{Example}
\newtheorem{exs}[thm]{Examples}

\newtheorem{notn}[thm]{Notation}

\newcommand\Om{\Omega}

\newcommand\del{\partial}
\newcommand\vp{\varphi}
\newcommand \ve{\varepsilon}

\newcommand\si{s^{-1}}

\newcommand{\cat}{\mathbf}

\newcommand{\ob}{\operatorname{Ob}}

\newcommand\ot{\otimes}

\newcommand\egal[2]{\overset {#1}{\underset {#2}\rightrightarrows }}
\newcommand\adjunct[2]{\overset {#1}{\underset {#2}\rightleftarrows }}

\newcommand\cof{\;\ar@{ >->}[r]}
\newcommand\ccof{\;\ar@{ >->}[rr]}

\newdir{ >}{{}*!/-10pt/\dir{>}}

\renewcommand{\Bar}{\mathscr B}
\newcommand{\U}{\mathscr E}
\renewcommand{\P}{\mathscr P}
\newcommand\Ho{\operatorname{Ho}}

\newcommand\G{\mathsf G}
\newcommand\W{\mathsf W}
\newcommand\ow{\overline\W}

\newcommand\hker{\backslash\negthinspace \backslash}
\newcommand\hquot{/\negthinspace/}

\begin{document}

\title{Normal and conormal maps in homotopy theory}

\author {Emmanuel D. Farjoun}
\email{farjoun@math.huji.ac.il}
\address{Department of Mathematics, Hebrew University of Jerusalem,
Givat Ram, Jerusalem 9190, Israel}

\author{Kathryn Hess}
\email{kathryn.hess@epfl.ch}
\address{MATHGEOM,
    \'Ecole Polytechnique F\'ed\'erale de Lausanne,
    CH-1015 Lausanne,
    Switzerland}

\begin{abstract} Let $\cat M$ be a monoidal category endowed with a distinguished class of weak equivalences and  with appropriately compatible classifying bundles for monoids and comonoids.  We define and study homotopy-invariant notions of normality for maps of monoids  and of conormality for maps of comonoids in $\cat M$.  These notions generalize both principal bundles and crossed modules and are preserved by nice enough monoidal functors, such as the normaliized chain complex functor.

 We provide   several explicit  classes of examples of homotopy-normal and of homotopy-conormal maps, when $\cat M$ is the category of simplicial sets or the category of chain complexes over a commutative ring.
\end{abstract}


\maketitle

\section*{Introduction}
In a recent article \cite{farjoun-segev}, it was observed that a crossed module of groups could be viewed as a up-to-homotopy version of the inclusion of a normal subgroup.  Motivated by this observation, the authors of \cite{farjoun-segev} formulated a definition of h-normal (homotopy normal) maps of simplicial groups, noting in particular that the conjugation map from a simplicial group into its simplicial automorphism group is h-normal.

The purpose of this paper is to provide a general, homotopical framework for the results in \cite{farjoun-segev}.  In particular, given a monoidal category $\cat M$ that is also a homotopical category \cite{dwyer-hirschhorn-kan-smith}, under reasonable compatability conditions between the two structures, we define and study notions of h-normality for maps of monoids and of h-conormality for maps of comonoids (Definition \ref{defn:normal}), observing that the two notions are dual in a strong sense.  Our work is complementary to that of Prezma in \cite{prezma}.

\subsection*{The motivating example} The following definition of homotopy normal maps 
of discrete groups or, more generally, of loop spaces, inspired our Definition \ref{defn:normal}. A map $n: N\to G$ of loop spaces was said in \cite {farjoun-segev} to be {\em homotopy normal}
if there exists a  connected, pointed space $W$  and a map $ w:BG\to W,$  its ``normality structure,'' such
 that the usual Nomura-Puppe sequence  obtained by
taking the Borel construction $EN\times_NG$, denoted $G\hquot N$ below, can be extended
to the right to form a  fibration sequence, all of whose terms are obtained from the map $w$ by repeatedly taking
homotopy fibres:
\begin{equation}\label{eqn:NP}\xymatrix{ N\ar[r]^n&G\ar[r]&G\hquot N\ar[r]&BN\ar[r]&BG\ar[r]^w&W.}\end{equation}

Once we have formulated the definition of h-normality and h-conormality, we provide several interesting classes of examples.  We first show that trivial extensions of monoids are h-normal and that trivial extensions of comonoids are h-conormal, at least under a reasonable extra condition on the category $\cat M$ (Proposition \ref{prop:trivialext}).  Considering the category of simplicial sets in particular, we then establish the following characterization (Propositions \ref{prop:looping-conormal}, \ref{prop:char-h-normal} and \ref{prop:allsimpl}), which shows that Definition \ref{defn:normal} does in fact generalize the motivating example.

\begin{thmA}
 \begin{enumerate}
\item A morphism $f$  of simplicial groups is h-normal if and only if there is a morphism of reduced simplicial sets $g:X\to Y$ such that $f$ is weakly equivalent to the (Kan) loops on $\operatorname{hfib}(g)\to X$, where $\operatorname{hfib}(g)$ denotes the homotopy fiber of $g$.
\item Any morphism of simplicial sets is h-conormal.
\end{enumerate}
\end{thmA}

Note that part (1) of the theorem above is analogous to the fact that the normal subgroups of a group $G$ are exactly the kernels of surjective homomorphisms with domain $G$.  In Remark \ref{rmk:bimonoid-kernel} we discuss the possibility of generalizing this result to other monoidal categories, pointing out some of the obvious difficulties and indicating how they might be avoided.

\subsection*{Chain complexes}
Motivated by the classical problem of determining the effect of applying various generalized homology theories to (co)fibration sequences,  we consider next the question of whether the normalized chains functor preserves h-(co)normality. The study of h-(co)normality in the category of chain complexes over a commutative ring is more delicate than in the simplicial case, as the tensor product of chain complexes is not the categorical product. It does turn out that, just as any subgroup of an abelian group is normal, any morphism of commutative chain algebras is h-normal;  the dual result holds as well (Proposition \ref{prop:abelian}).  Moreover, the normalized chains functor $C_{*}$ from simplicial sets to chain complexes preserves both h-normality and h-conormality, at least under mild connectivity conditions (Proposition \ref{prop:ch-conormal-induced} and Corollary \ref{cor:ch-normal-induced}).

\begin{thmB}  \begin{enumerate}
\item If $f: G\to G'$ is an h-normal morphism of reduced simplicial groups, then $C_{*}f:C_{*}G\to C_{*}G'$ is an h-normal morphism of chain algebras.
\item If $g:X\to Y$ is a simplicial morphism, where $Y$ is $1$-reduced and both $X$ and $Y$ are of finite type, then $C_{*}g: C_{*}X\to C_{*}Y$ is an h-conormal morphism of chain coalgebras.
\end{enumerate}
\end{thmB}

\subsection*{ Twisting structures} We formulate the concepts of h-normality  and h-conor\-mal\-ity within the framework of \emph{twisted homotopical categories}, i.e., categories endowed with both a monoidal and a homotopical structure, where the compatability between the two structures is mediated by a \emph{twisting structure} (Definition \ref{defn:twisting} and \ref{defn:twist}) that satisfies certain additional axioms (Definition \ref{defn:thc}).  The homotopical structure required is considerably less rigid than that of a Quillen model category, and the compatibility between the homotopical and monoidal structure is of a rather different nature from that of a monoidal model category: we require little more than the  capability to compute certain certain equalizers and coequalizers in a homotopy-meaningful manner.

 A twisting structure on a monoidal category is essentially a theory of principal and coprincipal bundles in the monoidal category, based on a looping/delooping (or cobar/bar) adjunction 
\[
\Om: \cat {comon} \rightleftarrows\cat {mon}: \Bar
\]
 between certain full subcategories of comonoids and of monoids in $\cat M$.  The required compatibility with homotopical structure insures that the total object of the classifying bundle of a monoid  or of a comonoid is homotopically trivial.

If $\cat M$ is a twisted homotopical category, then every monoid map $f:A\to A'$ induces an associated \emph{Nomura-Puppe sequence} (Definition \ref{defn:puppe})
\[
A\xrightarrow f A' \xrightarrow {\pi_{f}} A'\hquot A \xrightarrow {\delta _{f}} \Bar A \xrightarrow{\Bar f} \Bar A',
\]
where $A'\hquot A$ is a model for the homotopy quotient of $f$ built from the classifying bundle of $A$.  Similarly, every comonoid map $g:C'\to C$ induces an associated \emph{dual Nomura-Puppe sequence}
\[
\Om C'\xrightarrow {\Om g}\Om C \xrightarrow {\del_{g}} C\hker C'\xrightarrow {\iota _{g}} C' \xrightarrow{g} C,
\]
where $C\hker C'$ is a model for the homotopy kernel of $g$ built from the classifying bundle of $C$.  Our definition of h-normality and h-conormality is formulated in terms of these sequences (cf. Remark \ref{rmk:big-diagram}), whence the importance of the following result (Lemma \ref{lem:amusing}): for every monoid morphism $f:A\to A'$ and for every comonoid morphism $g:C'\to C$, there are  the  following corresponding commuting diagrams:
\[
\xymatrix {\Om \Bar A\ar [r]^{\Om \Bar f}\ar[d]^\sim&\Om \Bar A' \ar[d]_{}^\sim \ar[r]^(0.45){\del_{\Bar f}}& \Bar A'\hker \Bar A \ar [d]^\sim \ar[r]^(0.6){\iota_{\Bar f}}& \Bar A\ar[d]^=\ar[r]^{\Bar f}&\Bar A'\ar [d]^=\\
A\ar[r]^f&A' \ar [r]^{\pi_{f}}& A'\hquot A\ar [r]^{\delta_{f}}&\Bar A\ar[r]^{\Bar f}&\Bar A'}
\]
and
\[
\xymatrix{\Om C' \ar [d]^=\ar[r]^{\Om g}&\Om C \ar [d]^=\ar[r]^{\del_{g}}& C\hker C'\ar[d]^\sim \ar[r]^{\iota_{g}}&C'\ar[d]_{}^\sim\ar [r]^g&C\ar [d]^\sim\\
\Om C \ar[r]^{\Om g}&\Om C\ar [r] ^{\pi_{\Om g}}&\Om C\hquot \Om C' \ar[r] ^{\delta_{\Om g}}& \Bar \Om C'\ar [r]^{\Bar \Om g}&\Bar \Om C.}
\]
In other words, the Nomura-Puppe sequence of $f$ is weakly equivalent to the dual Nomura-Puppe sequence of $\Bar f$, and the dual Nomura-Puppe sequence of $g$ is weakly equivalent to the Nomura-Puppe sequence of $\Om g$.

We show in section \ref{sec:ex-twist-htpic} that both the category of reduced simplicial sets and the category of chain complexes over a commutative ring that are degreewise finitely generated projective are twisted homotopical categories.

The last two sections of this paper consist of appendices, containing necessary algebraic and homotopical preliminaries.  We recall the notions of twisting functions and twisting cochains in the first appendix.  We encourage the reader  with questions about the terminology and notation concerning simplicial sets and chain complexes used throughout this paper to consult this appendix. The second appendix consists of a review and further elaboration of the theory of twisting structures as developed in \cite{hess-lack} and \cite{hess-scott}, of which twisting functions and twisting cochains provide the primary examples.

\subsection*{Perspectives}

The following possible extensions and applications of the theory of h-normality and h-conormality should be interesting to study.
\begin{description}
\item [Functoriality] Under what conditions does a homotopical functor between twisted homotopical categories preserve h-normal and/or h-conormal morphisms?  In particular, is h-(co)normality invariant under localization, in the spirit of the recent article \cite {dwyer-farjoun} and Prezma's results in \cite {prezma}?
\item [Bimonoids] It should be possible to develop a rich theory of ``homotopy exact'' sequences of bimonoids in a twisted homotopical category, since a morphism of bimonoids can be h-normal or h-conormal.  We expect that  Segal structures should prove very useful to this end, in describing ``up to homotopy'' multiplicative and comultiplicative structures.
\item [Other operads] To what operads other than the associative operad can we extend the theory of h-normality of algebra morphisms over operads?  It seems likely that in the chain complex case, an analogous theory makes sense at least for any Koszul operad, built on the associated cobar/bar adjunction.
\item[Other categories]  What other interesting categories admit a reasonable twis\-ted homotopical structure? For example, it was proved in \cite{hess-scott} that the category of symmetric sequences of complexes that are degreewise finitely generated projective, endowed with its composition monoidal structure, is a twistable category admitting a twisting structure extending the usual cobar/bar adjunction, which is almost certainly appropriately compatible with its homotopical structure.  If so, then the study of normal morphisms of operads 
    should prove  interesting.
\end{description}

\subsection*{Related work}  In \cite{prezma} Prezma studied h-normal maps of (topological) loop spaces.  In particular, he characterized h-normal loop maps $\Om f: \Om X\to \Om Y$ as those for which there exists a simplicial loop space $\Gamma_{\bullet}$ such that $\Gamma_{0}\sim \Om Y$ and the canonical actions of $\Om Y$ on $\Gamma_{\bullet}$ and on $\operatorname{Bar}_{\bullet}(\Om X, \Om Y)$ are equivalent.  It follows from this characterization that if $L$ is an endofunctor of the category of topological spaces that preserves homotopy equivalences and such that the natural map $L(X\times Y) \to L(X)\times L(Y)$ is a homotopy equivalence for all $X$ and $Y$, then $L(\Om f)$ is h-normal whenever $\Om f$ is h-normal.

\subsection*{Notation and conventions}

\begin{itemize}
\item Let $\cat M$ be a small category, and let $A,B\in \ob \cat M$.  In these notes, the set of morphisms from $A$ to $B$ is denoted $\cat M(A,B)$.  The identity morphism on an object $A$ is often denoted $A$ as well.
\item If $X$ is a simplicial set, the $C_{*}X$ denotes the normalized chain complex on $X$, endowed with its usual coalgebra structure.
\item If $\Bbbk$ is a commutative ring, then $\cat {Ch}_{\Bbbk}$ is the category of chain complexes of $\Bbbk$-modules, while $\cat {dgProj}_{\Bbbk}$ is the full subcategory determined by  those complexes that are degreewise $\Bbbk$-projective and finitely generated.
\item The symbols $\eta$ and $\ve$ are always used to denote either the unit and augmentation of a monoid or the coaugmentation and counit of a comonoid.  The sense in which they are used on each occasion should be clear from context.
\end{itemize}

\section{Twisted homotopical categories}

In this section we consider categories admitting  twisting structures that are compatible in a reasonable way with a notion of weak equivalence. We refer the reader to Appendix \ref{sec:twisting} for a detailed introduction to twisting structures, where we fix all of the notation and terminology used in the remainder of the paper.

\subsection{Homotopical categories}

In \cite{dwyer-hirschhorn-kan-smith}, the authors developed and studied the following, minimalist framework for doing homotopy theory.

\begin{defn} A \emph{homotopical category} consists of a category $\cat M$ together with a distinguished class $\mathsf W$ of morphisms containing all identity maps and such that the ``two out of six'' property holds, i.e.,  for any sequence of morphisms
\[
W\xrightarrow r X\xrightarrow s Y\xrightarrow t Z
\]
in $\cat M$,
\[
sr, ts\in \mathsf W \Longrightarrow r,s,t, tsr\in \mathsf W.
\]
The elements of $\mathsf W$ are called \emph{weak equivalences} and denoted $\xrightarrow \sim$.

A functor between homotopical categories that preserves weak equivalences is also called \emph{homotopical}.
\end{defn}

A homotopical category has just enough structure to allow for a homotopy theory of its objects and morphisms, i.e.,  the following definition makes sense.

\begin{defn}  Let $(\cat M, \mathsf W)$ be a homotopical category.   Its \emph{homotopy category}, denoted $\Ho \cat M$, is the localization of $\cat M$ at $\mathsf W$.
\end{defn}

One can construct $\Ho\cat M$ explicitly, setting  $\ob \Ho \cat M=\ob \cat M$, while for all objects $X$ and $Y$,  $\Ho \cat M(X,Y)$ is the set of equivalence classes of zigzags
 of arrows in $\cat M$ linking $X$ to $Y$ such that every backward arrow is a weak equivalence, under the equivalence relation generated by omitting identity maps, replacing adjacent maps with the same orientation by their composite, and omitting pairs of adjacent maps with opposite orientation but the same label.

\begin{exs}
\begin{enumerate}
\item  The category of simplicial sets, endowed with its usual weak equivalences, is a homotopical category

\item The category of chain complexes over any ring is a homotopical category when $\mathsf W$ is chosen to be the class of all quasi-isomorphisms, i.e., chain maps inducing isomorphisms in homology.

\item Let $\cat D$ be a small category. If $(\cat M, \mathsf W)$ is a homotopical category, then the diagram category $\cat M^{\cat D}$ is as well, where the associated weak equivalences are the natural transformations $\tau: F\to G$ such that each component $\tau_{d}:F(d)\to G(d)$ is a weak equivalence, where $d\in \ob \cat D$.
\end{enumerate}
\end{exs}

We now specify the compatibility we require between homotopical structure and twisting structure.  Note that when we say that a morphism of structured objects in $\cat M$ (i.e., of monoids, comonoids, modules, etc.) is a weak equivalence, we mean that the underlying morphism in $\cat M$ is a weak equivalence.

\begin{rmk}  Let $(\cat M, \mathsf W)$ be a homotopical category, which is also endowed with a monoidal structure.  We can then define the structure of a homotopical category on $\cat {Bun}$ by saying that a morphism
\[
\xymatrix {A\ar [d]_{\alpha}\ar [r]^i&N\ar [d]_{\gamma}\ar [r]^p &C\ar [d]_{\beta}\\ A'\ar [r]^{i'}&N'\ar [r]^{p'} &C'}
\]
of mixed bundles is a weak equivalence if $\alpha$, $\beta$,  and $\gamma$ are all weak equivalences.
\end{rmk}

\subsection{Twisting structures and  weak equivalences} We now formulate the main definition of this section, which consists of
a sequence of compatibility  requirements between the given twisting structure and  the class of weak equivalences.

\begin{defn}\label{defn:thc}  A \emph{twisted homotopical category} is a homotopical category $(\cat M, \mathsf W)$ such $\cat M$ also admits a monoidal stucture $(\cat M, \otimes , I)$ endowed with a twistable triple $(\cat {mon}, \cat {comon}, \operatorname{mix})$ and a twisting structure  $(\Om, \Bar, \zeta, \xi)$, which is compatible with $\mathsf W$ in the following sense.
\begin{enumerate}
\item the unit and counit of the $(\Om , \Bar)$-adjunction are natural weak equivalences;
\item $\Om $ and $\Bar$ are both homotopical functors;
\item for all $A\in \ob \cat {mon}$, the composite
\[
\U A \xrightarrow {q_{A}} \Bar A \xrightarrow {\ve_{\Bar A}} I
\]
is a weak equivalence;
\item for all $C\in \cat{comon}$, the composite
\[
I\xrightarrow {\eta _{\Om C}} \Om C \xrightarrow {i_{C}} \P C
\]
is a weak equivalence;
\item  if $f:A\xrightarrow \sim A'$ is weak equivalence in $\cat {mon}$, and $\zeta=(A\to N\to C)$ is any classifiable biprincipal bundle, then the natural map $\zeta \to f_{*}(\zeta)$ (cf. Remark \ref{rmk:bndlmap}) is a weak equivalence of bundles; and
\item if $g:C'\xrightarrow \sim C$ is weak equivalence in $\cat {comon}$, and $\zeta=(A\to N\to C)$ is any classifiable biprincipal bundle, then the natural map $g^*(\zeta) \to \zeta$ (cf. Remark \ref{rmk:bndlmap}) is a weak equivalence of bundles.
\end{enumerate}

\end{defn}

\begin{rmk}  To simplify the presentation henceforth, we denote twisted homotopical categories simply as $\cat M$ and suppress explicit mention of the rest of the structure in the notation.
\end{rmk}

\begin{rmk} Since $q_{A}:\U A \to \Bar A$ and $\ve_{\Bar A}:\Bar A\to I$ are both morphisms of right $A$-modules (cf. Remark \ref{rmk:biprin}), condition (3) says that $\U A$ is resolution of $I$ as a right $A$-module.  Similarly, condition (4) implies that $\P C$ is a resolution of $I$ as a left $C$-comodule.
\end{rmk}

\begin{rmk}\label{rmk:Borel-inv} Condition (5) implies that a commutative diagram in $\cat {mon}$
\[
\xymatrix {A \ar[d]_{}^\alpha_{\sim}\ar [r]^f& A'\ar [d]_{}^{\alpha'}_{\sim}\\ B\ar [r]^{g}& B'},
\]
in which the vertical arrows are weak equivalences, induces a weak equivalence of Borel quotients
\[
\alpha'\hquot \alpha: A'\hquot A \xrightarrow\sim B'\hquot B,
\]
since $(\alpha, \alpha')_{*}: f_{*}\big( \zeta (A)\big) \to g_{*}\big( \zeta(B)\big)$  can be identified with the composite
\[
f_{*}\big(\zeta (A)\big) \xrightarrow {\sim} \alpha'_{*}f_{*}\big(\zeta (A)\big)\cong (\Bar \alpha)^*(\Bar g)^*\big(\zeta (B')\big)\xrightarrow \sim (\Bar g)^*\big(\zeta (B')\big).
\]
The factorization above depends heavily on the twistability of $(\cat {mon}, \cat {comon}, \operatorname{mix})$, as well as on axiom (2) of Definition \ref{defn:twist}, which together give rise to a sequence of isomorphisms
{\small \[
 \alpha'_{*}f_{*}\big(\zeta (A)\big)\cong(\alpha'f)_{*}\big(\zeta (A)\big)=(g\alpha)_{*}\big(\zeta (A)\big)\cong \big(\Bar (g\alpha)\big)^*\big(\zeta (B')\big)\cong (\Bar \alpha)^*(\Bar g)^*\big(\zeta (B')\big).
\]}

We are therefore justified in regarding the Borel quotient of a monoid map in a twisted homotopical category as its \emph{homotopy quotient}, as the construction is an invariant of weak equivalence, especially since condition (2)  can be understood to mean that $\U A$ is right $A$-module resolution of $I$, as in the previous remark.

Similarly, it follows from condition (6) that a commutative diagram in $\cat {comon}$
\[
\xymatrix {C \ar[d]_{}^\beta_{\sim}\ar [r]^f& C'\ar [d]_{}^{\beta'}_{\sim}\\ D\ar [r]^{g}& D'},
\]
in which the vertical arrows are weak equivalences, induces a weak equivalence of Borel kernels
\[
\beta \hker \beta ': C\hker C' \xrightarrow \sim D\hker D',
\]
whence our vision of the Borel kernel as a \emph{homotopy kernel}, when working in a twisted homotopical category.  Interpreting condition (3) to mean that $\P C$ is a left $C$-comodule resolution of $I$, as in the previous remark, further substantiates this vision.
\end{rmk}

\begin{rmk}  If $\cat M$ is a twisted homotopical category, then so is the category $\cat M^{\cat D}$ of $\cat D$-shaped diagrams in $\cat M$, for any small category $\cat D$.  All required structure is defined objectwise in $\cat {M^D}$.
\end{rmk}

The lemma below, which relates Nomura-Puppe and dual Nomura-Puppe sequences, is amusing in its own right and proves very helpful in understanding the relationship between normality and conormality up to homotopy in the next section.

\begin{lem}\label{lem:amusing} Let $\cat M$ be a twisted homotopical category.
\begin{enumerate}
\item For every morphism $f:A\to A'$ in $\cat {mon}$, there is a weak equivalence of biprincipal bundles
\[
\xymatrix {\Om \Bar A' \ar[d]_{v_{A}}^\sim \ar[r]& \Bar A'\hker \Bar A \ar [d]^\sim \ar[r]& \Bar A\ar[d]^=\\
A' \ar [r]& A'\hquot A\ar [r]&\Bar A.}
\]
\item For every morphism $g:C'\to C$ in $\cat {comon}$, there is a weak equivalence of biprincipal bundles
\[
\xymatrix{\Om C \ar [d]^=\ar[r]& C\hker C'\ar[d]^\sim \ar[r]&C'\ar[d]_{u_{C}}^\sim\\
\Om C\ar [r] &\Om C\hquot \Om C' \ar[r] & \Bar \Om C'.}
\]
\end{enumerate}
\end{lem}

\begin{proof} We prove (1) and leave the dual proof of (2) to the reader.
There is a diagram of morphisms of biprincipal bundles
{\small \[
\xymatrix{&&&\Bar A'\hker \Bar A\ar@{=}[d]\ar[dr]\\
(\Bar f)^*\big(\xi (\Bar A')\big)\ar [d]^\sim&=&\big( \Om\Bar A' \ar[r]\ar[d]_{v_{A}}^\sim\ar[ur]&\Bar A\square_{\Bar A'}\P\Bar A'\ar [r]\ar [d]^\sim &\Bar A\ar [d]^=\big)\\
(v_{A})_{*}(\Bar f)^*\big(\xi (\Bar A')\big)\ar [d]^\cong&=&\big( A' \ar[r]\ar[d]_{=}&\Bar A\square_{\Bar A'}\P\Bar A'\ot_{\Om\Bar A'}A'\ar [r]\ar [d]_{=} &\Bar A\ar [d]^=\big)\\
(\Bar f)^*\big(\zeta (A')\big)\ar[d]^\cong&=&\big( A' \ar[r]\ar[d]_{=}&\Bar A\square_{\Bar A'}\U A'\ar [r]\ar [d]_{\cong} &\Bar A\ar [d]^=\big)\\
f_{*}\big(\zeta (A)\big)&=&\big( A' \ar[r]\ar[dr]&\U A\ot_{A}A'\ar [r]\ar@{=}[d]&\Bar A.\big)\\
&&&A'\hquot A\ar[ur]}
\]}

The topmost vertical arrow is a weak equivalence by axioms (1) and (5) of Definition \ref{defn:thc}.  The second vertical arrow is the isomorphism obtained by applying $(\Bar f)^*$ to the isomorphism guaranteed by axiom (1) of Definition \ref{defn:twist} and recalling that $(v_{A})_{*}(\Bar f)^*=(\Bar f)^*(v_{A})_{*}$ (Remark \ref{rmk:commute}). Finally, the third vertical arrow is exactly the isomorphism guaranteed by the axiom (2) of Definition \ref{defn:twist}.
\end{proof}

\subsection{Examples}\label{sec:ex-twist-htpic}
\subsubsection{Simplicial sets}  It is well known that the unit and counit of the $(\G, \overline\W)$-adjunction are natural weak equivalences, that $\overline\W$ and $\G$ are homotopical, and that $\overline \W G\times_{ {\nu_{G}}} G$ and $X\times_{{\tau_{X}}}\G X$ are both acyclic \cite{may}.  Finally,  we can apply the long exact homotopy sequence of a twisted cartesian product to prove conditions (5) and (6) of Definition \ref{defn:thc}, as follows.

Let $f:G\to G'$ be a weak equivalence of simplicial groups, and let 
\[
\zeta=(G\to X\times_{\nu _{G}g} G\to X)
\]
 be a biprincipal bundle classified by a simplicial map $g:X\to \overline \W G$. The natural morphism of bundles from $\zeta$ to $f_{*}(\zeta)$ is then
\[
\xymatrix{G\ar [r]\ar [d]_{f}^\sim& X \times_{{\nu_{G}g}} G \ar[r] \ar [d] _{X\times f}&X\ar [d]_{}^=\\
G'\ar [r]&X  \times_{f\nu_{G}g} G' \ar[r] & X.}
\]
Since biprincipal bundles are, in particular, twisted cartesian products, and two of the three vertical maps are weak equivalences, the third must be as well, proving condition (5).  Condition (6) is proved similarly.

\subsubsection{Chain complexes}  The unit and counit of the $(\Om, \Bar)$-adjunction are well known to be natural weak equivalences, just as $\Om$ and $\Bar $ are well known to be homotopical.  Moreover, $\Bar A\otimes _{{t_{\Bar}}}A$ and $C\otimes_{{t_{\Om}}} \Om C$ are, respectively, the acyclic bar and acyclic cobar constructions, which are easily seen to be contractible.  Finally, conditions (5) and (6) can be proved by a simple argument using Zeeman's comparison theorem, as follows.

Let $f:A\to A'$ be a quasi-isomorphism of chain algebras, and let 
\[
\zeta=(A\to C\otimes_{{t_{\Bar}g}} A \to C)
\]
 be a biprincipal bundle classified by a chain coalgebra map $g:C\to \Bar A$.  The natural morphism of bundles from $\zeta$ to $f_{*}(\zeta)$ is then
\[
\xymatrix{A\ar [r]\ar [d]_{f}^\sim& C\otimes_{{t_{\Bar}g}} A\ar[r] \ar [d] _{C\otimes f}&C\ar [d]_{}^=\\
A'\ar [r]& C \otimes_{{ft_{\Bar}g}}  A' \ar[r] &C.}
\]
Applying Zeeman's comparison theorem to the obvious filtration on the total object of each bundle, we see that since  two of the three vertical maps are weak equivalences, the third must be as well, proving condition (5).  Condition (6) is proved similarly.

\section{H-normal and h-conormal maps}

In this section we define homotopical notions of normality and conormality in a twisted homotopical category, inspired by the example of the Nomura-Puppe sequence for loop spaces and discrete groups (\ref{eqn:NP}).
We study the elementary properties of these notions and provide explicit examples of classes of such maps in the categories of simplicial sets and of chain complexes.

Throughout this section we suppose that $\cat M$ is a twisted homotopical category, with respect to some fixed twistable triple $(\cat {mon}, \cat {comon}, \operatorname{mix})$.  All homotopy quotients and homotopy kernels are defined in terms of the fixed twisting structure $(\Om, \Bar, \zeta, \xi)$ on $\cat M$.

\subsection{Definitions and elementary properties}

Our definition of homotopy nor\-mality is motivated by the fact that a subgroup $N$ of a group $G$ is normal if and only if the quotient set $G/N$ of orbits of the $N$-action on $G$ admits a group structure such that the quotient map $G\to G/N$ is a homomorphism.  We must, however, replace quotients by homotopy quotients and consider multiplicative structure up to homotopy, in some suffciently rigid sense.

We formulate h-normality in terms of the following sorts of sequences.

\begin{defn} An \emph{extended bundle} in $\cat M$ is a sequence
\[
A\xrightarrow j M \xrightarrow d N \xrightarrow p C
\]
of morphisms in $\cat M$, where
\begin{itemize}
\item $A\in \ob \cat{mon}$;
\item $M$ is a right $A$-module, and $j$ is a morphism of $A$-modules;
\item $C\in \ob \cat {comon}$; and
\item $N$ is a left $C$-comodule, and $p$ is a morphism of $C$-comodules.
\end{itemize}
A morphism of extended bundles  consists of a commuting diagram in $\cat M$
\[
\xymatrix{A \ar [d]_{\alpha}\ar[r]^{j}&M\ar [d]_{\mu}\ar[r]^d&N\ar[d]_{\nu}\ar[r]^p&C\ar[d]_{\beta}\\
A'\ar[r]^{j'}&M'\ar[r]^{d'}&N'\ar[r]^{p'}&C'}
\]
in which the rows are extended bundles, $\alpha$ is a monoid morphism, $\mu$ is a morphism of $A$-modules, $\beta$ is a comonoid morphism, and $\nu$ is a morphism of $C'$-comodules.

 A morphism $(\alpha, \mu, \nu, \beta): \tau \to \tau'$ of extended bundles is an \emph{elementary equivalence}, denoted $\tau \xrightarrow \sim \tau'$, if every component is a weak equivalence.  An \emph{equivalence of extended bundles} is a zigzag of elementary equivalences.
\end{defn}

\begin{rmk}  Note that the morphism $d$ in the definition of extended bundle is assumed only to be a morphism in $\cat M$.
\end{rmk}

The main examples of extended bundle  sequences arise from the Nomura-Puppe and dual Nomura-Puppe sequences (Definition \ref{defn:puppe}).

\begin{ex} For any morphism $f:A\to A'$ in $\cat{mon}$, the extended bundle
\[
\tau(f)=(A' \xrightarrow {\pi_{f}} A'\hquot A \xrightarrow {\delta_{f}} \Bar A \xrightarrow {\Bar f} \Bar A')
\]
is the \emph{truncated Nomura-Puppe sequence} associated to $f$.
Dually, for any morphism $g:C'\to C$ in $\cat{comon}$, the extended bundle
\[
\theta(g)=(\Om C' \xrightarrow {\Om g} \Om C \xrightarrow {\del _{g}} C\hker C' \xrightarrow {\iota_{g}} C')
\]
is the \emph{truncated dual Nomura-Puppe sequence} associated to $g$.
\end{ex}

We can now formulate our definition of homotopy normality and conormality.

\begin{defn} \label{defn:normal}{\em Let $f:A\to A'$ and $g:C'\to C$  be morphisms in $\cat{mon}$ and $\cat {comon}$, respectively.  The pair $(f,g)$ is \emph{normal} if  the extended bundles $\tau(f)$ and $\theta(g)$ are equivalent.

If $(f,g)$ is a normal pair, then $f$ is \emph{h-normal} with associated \emph{normality structure} $g$, while $g$ is \emph{h-conormal} with associated \emph{conormality structure} $f$.}
\end{defn}
\smallskip
\begin{rmk}\label{rmk:big-diagram} The definition above implies that if $(f,g)$ is a normal pair, then the dual Nomura-Puppe sequence of $g$ and the Nomura-Puppe sequence of $f$ (the second and next-to-last columns below) fit into the following sort of commutative diagram.

\[
\xymatrix{&&&&&A\ar[d]_{f}\\
\Om C'\ar [d]^{\Om g}&\Om C'\ar [d]^{\Om g}\ar [l]^=&\bullet \ar[d]\ar [l]^\sim \ar [r]_{\sim}&\cdots&\bullet \ar[d]\ar [l]^\sim\ar [r]_{\sim}&A'\ar [d]_{\pi_{f}} &\Om \Bar A'\ar [l]^\sim\ar[d]_{\del_{\Bar f}}\\
\Om C\ar[d]^{\pi_{\Om g}}&\Om C\ar [d]^{\del_{g}}\ar [l]^=&\bullet\ar[d] \ar [l]^\sim \ar [r]_{\sim}&\cdots&\bullet \ar[d]\ar [l]^\sim\ar [r]_{\sim}&A'\hquot A\ar [d]_{\delta_{f}}&\Bar A'\hker \Bar A\ar [d]_{\iota_{\Bar f}}\ar[l]^\sim\\
\Om C\hquot \Om C'\ar[d]^{\delta_{\Om g}}&C\hker C'\ar [d]^{\iota _{g}}\ar [l]^\sim&\bullet\ar[d]\ar [l]^\sim\ar [r]_{\sim}&\cdots&\bullet\ar[d] \ar [l]^\sim\ar [r]_{\sim}&\Bar A\ar [d]_{\Bar f}&\Bar A\ar[l]^=\ar [d]_{\Bar f}\\
\Bar \Om C'&C'\ar[d]_{g}\ar [l]^\sim&\bullet \ar [l]^\sim \ar [r]_{\sim}&\cdots&\bullet \ar [l]^\sim\ar [r]_{\sim}&\Bar A'&\Bar A'\ar [l]^=\\
&C.}
\]
Note that we have applied Lemma \ref{lem:amusing} to obtain  the first and last columns of equivalences above.

Looking at the second and fifth columns (from the right) of this diagrams, we see that a monoid morphism $f:A\to A'$ is h-normal if  its associated Nomura-Puppe sequence can be recovered, up to weak equivalence, from the dual Nomura-Puppe sequence of a comonoid map.  In particular, the Borel quotient $A'\hquot A$ of $f$ is weakly equivalent to a monoid.

Dually, a comonoid morphism $g$ is h-conormal if its associated dual Nomura-Puppe sequence can be recovered, up to homotopy, from the Nomura-Puppe sequence of a monoid map.  In particular, the Borel kernel $C\hker C'$ of $g$ is weakly equivalent to a comonoid.

\end{rmk}

\begin{rmks}\label{rmks:norm} \begin{enumerate}
\item Though we do not provide a general treatment of functoriality (namely, behavior under monoidal functors) of h-normality and h-conormality in this paper, the relation between the simplicial and chain complex examples treated later in this section gives a indication of the sort of results that can be expected. In general,
    however, if  $(f,g)$ is a normal pair $\cat M$  and  $L:\cat M\to \cat M'$ is a monoidal, homotopical functor between twisted homotopical categories, it may happen that  the  map $Lf$ is normal, though its normality
    structure might   {\em not}  be $Lg.$
\item It is an immediate consequence of Definition \ref{defn:normal} that  h-normality and h-conormality are strictly  dual notions.
\item Remark \ref{rmk:Borel-inv} implies that h-normality and h-conormality are homotopy invariant notions.
\item There may be many, nonequivalent (co)normality structures associated to a given h-(co)normal map.  For example, in the case of discrete groups, several distinct crossed module structures may correspond to the same homomorphism $N\to G.$

\end{enumerate}
\end{rmks}

We next prove a result giving a sufficient condition for h-normality, which clarifies somewhat the relation between our definition of h-normality and that of normal subgroups: if the homotopy quotient of a monoid morphism $f$ itself admits a ``nice enough'' monoid structure, then $f$ is h-normal.  Note that this implication is usually not reversible, as the homotopy quotient of an h-normal map is supposed only to be weakly equivalent as a module to a monoid.    The existence of a such a ``nice'' monoid structure on the homotopy quotient can thus be viewed as a strong variant of h-normality.
\begin{lem}\label{lem:rigidnormal} Let $f:A\to A'$ be a morphism in $\cat {mon}$.  If $A'\hquot A$ admits a monoid structure with respect to which $\pi_{f}:A'\to A'\hquot A$ is a monoid morphism, and there is a weak equivalence of left $\Bar A'$-comodules $\tilde\pi:(A'\hquot A)\hquot A' \xrightarrow\sim \Bar A$ such that
\[
\xymatrix{A'\hquot A \ar[d]_{=}\ar[r]^(0.4){\pi_{\pi_{f}}}& (A'\hquot A)\hquot A'\ar [d]_{\sim}^{\tilde \pi}\ar[r]^(0.6){\delta_{\pi_{f}}}&\Bar A'\ar[d]_{=}\\
A'\hquot A \ar [r]^{\delta_{f}}& \Bar A\ar [r]^{\Bar f}& \Bar A'}
\]
commutes, then $f$ is h-normal with normality structure $\Bar \pi_{f}:\Bar A'\to \Bar (A'\hquot A)$.
\end{lem}
\begin{proof} Recall from Lemma \ref{lem:amusing} that there is a weak equivalence of $\Bar A'$-comodules
$\Bar (A'\hquot A)\hker \Bar A' \xrightarrow \sim (A'\hquot A)\hquot A'$, fitting into a weak equivalence of biprincipal bundles.
There is therefore an equivalence of extended bundles
\[
\xymatrix{\Om \Bar A'\ar [r]^{v_{A'}}_{\sim}\ar [d]_{\Om \Bar \pi_{f}}&A'\ar [r]_= \ar [d]_{\pi_{f}}& A'\ar [d]^{\pi_{f}}\\
\Om \Bar (A'\hquot A)\ar [d]_{\del_{\Bar \pi_{f}}}\ar [r]^{v_{A'\hquot A}}_{\sim}&A'\hquot A\ar[r]_=\ar [d]_{\pi_{\pi_{f}}}&A'\hquot A \ar [d]^{\delta _{f}}\\
\Bar (A'\hquot A)\hker \Bar A'\ar [d]_{\iota_{\Bar \pi_{f}}}\ar [r]_{\sim}&(A'\hquot A)\hquot A' \ar [r]_{\sim}^{\tilde \pi}\ar[d]_{\delta_{\pi_{f}}}&\Bar A \ar [d]^{\Bar f}\\
\Bar A'\ar [r]^=&\Bar A'\ar [r]^=&\Bar A',}
\]
implying that $f$ is h-normal with normality structure $\Bar \pi_{f}:\Bar A'\to \Bar(A'\hquot A)$.
\end{proof}

The dual condition, which we prove holds in the simplicial case in Lemma \ref{lem:simpl-good}, is formulated as follows; we leave its strictly dual proof to the reader.

\begin{lem}\label{lem:rigidconormal}  Let $g:C'\to C$ be a morphism in $\cat{comon}$.  If $C\hker C'$ admits a comonoid structure with respect to which $\iota _{g}:C\hker C' \to C'$ is a morphism of comonoids, and there is a weak equivalence of $\Om C'$-modules $\tilde\iota:\Om C\to C'\hker (C\hker C')$ such that
\[
\xymatrix{\Om C' \ar[d]_{=}\ar [r]^{\Om g}&\Om C\ar [d]_{\sim}^{\tilde \iota}\ar[r]^{\del_{g}}& C\hker C'\ar [d]_{=}\\
\Om C'\ar[r]^(0.4){\del_{\iota_{g}}}&C'\hker(C\hker C')\ar[r]^(0.6){\iota_{\iota_{g}}}&C\hker C'}
\]
commutes, then $g$ is h-conormal with conormality structure $\Om \iota _{g}:\Om (C\hker C')\to \Om C'$.
\end{lem}

\begin{rmk}[\emph{Normality of homotopy kernel maps}]\label{rmk:bimonoid-kernel} Continuing the comparison between the notions of normal subgroup and of h-normal monoid morphism, and motivated by the characterization of normal subgroups as kernels, we can ask under what conditions the homotopy kernel of a morphism of \emph{bimonoids} is h-normal.  Note that  we cannot even formulate such a question for a morphism of monoids, since the construction of the homotopy kernel requires comultiplicative structure.  The comultiplicative structure comes for free in the case of group homomorphisms, since the diagonal map endows any group with the structure of a bimonoid in the category of sets.

It is not difficult to see that if $H$ is a bimonoid in $\cat M$ such that $\P H$ admits a monoid structure with respect to which both its right $\Om H$-action and left $H$-coaction, as well as  the morphisms in the biprincipal bundle
\[
\xi(H)=(\Om H \xrightarrow {i_{H}} \P H \xrightarrow{p_{H}}H),
\]
are monoid morphisms,  then for all bimonoid maps $g: H'\to H$, the induced morphism 
\[
\iota _{g}:H\hker H' \to H'
\]
 is a monoid morphism.  In this case, it makes sense therefore to ask when $\iota_{g}$ is h-normal.

Dually if $H$ is a bimonoid in $\cat M$ such that $\U H$ admits a comonoid structure with respect to which both its right $H$-action and left $\Bar H$-coaction, as well as the morphisms in the biprincipal bundle
\[
\zeta(H)=(H \xrightarrow {j_{H}} \U H \xrightarrow{q_{H}}\Bar H)
\]
are comonoid morphisms, then for all bimonoid maps $f: H\to H'$, the induced morphism 
\[
\pi _{f}: H' \to H'\hquot H
\]
 is a comonoid morphism.  The question of when $\pi_{f}$ is h-conormal is thus meaningful in this case.

It follows from Corollary 3.6 in \cite{hess-levi} and its obvious dual that the conditions on $\P H$ and $\U H$ formulated above hold when $\cat M=\cat {dgProj}_{\Bbbk}$ and $H$ is any connected chain Hopf algebra.  We intend to study h-normality of homotopy kernels and h-conormality of homotopy quotients in $\cat {dgProj}_{\Bbbk}$ and similar categories in an upcoming paper.
\end{rmk}

\subsection{Examples}

We now present examples of h-normal and h-conormal maps.  We begin with a class of examples in a general twisted homotopical category, then consider the particular cases of simplicial sets and of chain complexes.

\subsubsection{A general class of examples}

We show that ``trivial extensions'' of monoids are h-normal, while ``trivial extensions'' of comonoids are h-conormal, at least under an additional hypothesis on $\cat M$.

\begin{prop}\label{prop:trivialext} Let $\cat M$ be a twisted homotopical category, with twisting structure $(\Om, \Bar, \zeta, \xi)$. If there are natural weak equivalences
\[
\Om (-\otimes -) \overset\sim\Longrightarrow \Om (-) \otimes \Om (-): \cat {Comon}^{\times 2} \to \cat {Mon}
\]
and
\[
\Bar (-) \otimes \Bar (-)\overset\sim\Longrightarrow \Bar (-\otimes -):\cat {Mon}^{\times 2} \to \cat {Comon},\]
 then
 \[
A\otimes \eta: A \to A\otimes B
\]
is h-normal with normality structure $\ve\otimes \Bar B :\Bar A\otimes \Bar B \to \Bar B$, for any augmented monoids $A$ and $B$, while
\[
\ve \otimes D: C\otimes D \to D
\]
is h-conormal with conormality structure $\Om C\otimes \eta: \Om C \to \Om C \otimes \Om D$, for any coaugmented comonoids $C$ and $D$.
\end{prop}

\begin{rmk}  It is well known that the necessary hypotheses on $\Om$ and $\Bar$ hold when $\cat M$ is either $\cat {sSet}$ or $\cat {dgProj}_{\Bbbk}$.
\end{rmk}

\begin{proof}   We prove the h-conormality of $\ve\otimes D$ and leave the other half of the proof, which is strictly dual, to the reader.

Observe that
\[
D\hker (C\otimes D) = (C\otimes D) \square _{D}\P D\cong C\otimes \P D
\]
and
\[
(\Om C\otimes \Om D)\hquot \Om C =\U \Om C \otimes _{\Om C} (\Om C \otimes \Om D)\cong \U \Om C \otimes \Om D
\]
as left $C$-comodules and right $\Om D$-modules.
There is therefore an equivalence of extended bundles
\[
\xymatrix{\Om (C\otimes D)\ar [r]^{}_{\sim}\ar [d]_{\Om (\ve\otimes D)}&\Om C\otimes \Om D\ar [r]_= \ar [d]^{}& \Om C\otimes \Om D\ar [d]^{\pi_{\Om C\otimes \eta}}\\
\Om D\ar [d]_{\del_{\ve\otimes D}}\ar [r]^{}_(0.4){\sim}&\U \Om C\otimes \Om D\ar[r]_(0.45)\cong\ar [d]_{}&\Om (C\otimes D)\hquot \Om C \ar [d]^{\delta _{\Om C\otimes \eta}}\\
D\hker (C\otimes D)\ar [d]_{\iota_{\ve\otimes D}}\ar [r]_{\cong}&C\otimes \P D \ar [r]_{\sim}^{}\ar[d]_{}&\Bar \Om C \ar [d]^{\Bar (\Om C\otimes \eta)}\\
C\otimes D\ar [r]^=&C\otimes D\ar [r]^{v_{C\otimes D}}_{\sim}&\Bar \Om (C\otimes D),}
\]
whence $\ve\otimes D$ is h-conormal with conormality structure $\Om C\otimes \eta$.
\end{proof}

\subsubsection{The simplicial case}

The work described in this article originated with following example.

\begin{ex} \label{ex:farjoun-segev} Let  $n:N\to G$ be a group homomorphism, seen as a morphism of constant simplicial groups. It was  proved in \cite{farjoun-segev} that, if there is a crossed module structure (cf. section 2.5 of \cite{farjoun-segev}) on the homomorphism $n$, then
$n$ is h-normal, in the sense defined here.  In particular, if $n$ is the inclusion of a normal subgroup, then it is h-normal.

In fact, the authors of \cite{farjoun-segev} proved an equivalence between the existence of a crossed module structure structure on $n$ and the existence of what they called a \emph{normal} simplicial group structure on the simplicial bar model of $G\hquot N$, in which the simplicial set is a free $G$-set in each level.
\end{ex}

\begin{ex}\label{ex:simpl} Let $e$ denote the trivial simplicial group, and let $G$ be any simplicial group.  The unique simplicial homomorphism $G\to e$ is h-normal if and only if  there is a simplicial set $X$ such  that $\overline\W G \sim \G X$ as simplicial sets. For example, if there is some $1$-reduced simplicial set $Y$ such that $G=\G^2Y$, then $G\to e$ is h-normal. Since double loop spaces are in some sense the generic homotopy commutative monoids, we see that h-normality of $G\to e$ is strongly related to homotopy commutativity of $G$.\end{ex}

More generally, homotopy fibers of simplicial maps give rise to h-normal maps of simplicial groups, as we show below.  The present example  is a reformulation of a well-known theorem about the classical dual Nomura-Puppe sequence of a continuous map, which uses the $A_{\infty}$-structure of based loop spaces. Our  proof here is quite simple and requires neither a full model category structure, nor higher homotopies: the twisting structure suffices.

\begin{prop}\label{prop:looping-conormal}  Let $g: X\to Y$ be a simplicial map, where $X$ and $Y$ are reduced.  If $\iota _{g}: Y\hker X\to X$ is the natural map from the homotopy fiber of $g$ to $X$, then the induced map of simplicial groups
\[
\G \iota _{g}: \G (Y\hker X)\to \G X
\]
is h-normal.
\end{prop}

The main step in proving this proposition is to show that the hypothesis of Lemma \ref{lem:rigidconormal} holds for all simplicial morphisms.  Since we apply this result elsewhere in this article as well, we state it as a separate lemma.

\begin{lem}\label{lem:simpl-good} If $g:X\to Y$ is any simplicial morphism, where $X$ and $Y$ are reduced, then there is a weak equivalence of $\G X$-modules $\tilde\iota:\G Y\to X\hker (Y\hker X)$ such that
\[
\xymatrix{\G X \ar[d]_{=}\ar [r]^{\G g}&\G Y\ar [d]_{\sim}^{\tilde \iota}\ar[r]^{\del_{g}}& Y\hker X\ar [d]_{=}\\
\G X\ar[r]^(0.4){\del_{\iota_{g}}}&X\hker(Y\hker X)\ar[r]^(0.6){\iota_{\iota_{g}}}&Y\hker X}
\]
commutes.
\end{lem}

\begin{proof} Observe that
\[
X\hker (Y\hker X) = (Y\hker X) \times_{\tau_{X}\iota _{g}} \G X =(X\times_{\tau_{Y}g} \G Y)\times_{\tau_{X}\iota _{g}} \G X.
\]
An easy calculation shows that, since $\iota_{g}:X\times_ {\tau_{Y}g}\G Y \to X$ is simply projection onto $X$, the map
\[
(X\times_{\tau_{Y}g}\G Y)\times_{\tau_{X}\iota _{g}} \G X\to  (X\times_ {\tau_{X}} \G X)\times \G Y: (x,v,w)\mapsto (x, w, \G(g) (w)^{-1}\cdot v)
\]
is a simplicial isomorphism.
Moreover, the inclusion of $ \G Y$ into $(X\times_{\tau_{X}}\G X)\times \G Y$ is  a weak equivalence, since $X\times_{\tau_{X}} \G X$ is contractible.  We can therefore set $\tilde \iota$ equal to the composite
\[
\G Y\overset \sim\hookrightarrow(X\times_ {\tau_{X}}\times \G X)\times \G Y \cong    X\hker(Y\hker X).
\]
\end{proof}

\begin{proof}[Proof of Proposition \ref{prop:looping-conormal}]  There is an equivalence of extended bundles
\[
\xymatrix{\G\ow \G X \ar [d]_{\G\ow \G g}\ar[r]^{v_{\G X}}_{\sim}& \G X\ar [d]_{\G g}& \G X\ar [l]_{=}\ar [d]_{\G g}\ar[r]^=&\G X \ar[d] _{\del_{\iota_{g}}}\ar [r]^=& \G X\ar [d]_{\pi_{\G \iota _{g}}}\\
\G\ow \G Y\ar [d]_{\del_{\ow \G g}}\ar[r]^{v_{\G Y}}_{\sim}& \G Y\ar[d]_{\pi_{\G g}}& \G Y\ar [l]_{=}\ar [d]_{\del_{g}}\ar [r]^(0.35){\tilde \iota}_(0.35){\sim}&X\hker (Y\hker X)\ar [d]_{\iota_{\iota_{g}}}\ar [r]_(0.45){\sim}& \G X\hquot \G (Y\hker X)\ar [d]_{\delta_{\G \iota _{g}}}\\
\ow \G Y\hker \ow \G X\ar [d]_{\iota _{\ow\G g}}\ar [r]_{\sim}& \G Y\hquot \G X\ar [d]_{\delta _{\G g}}&Y\hker X \ar [l]^\sim \ar[d]_{\iota_{g}}\ar [r]^=&Y\hker X\ar [d]_{\iota _{g}}\ar[r]^(0.4){u_{Y\hker X}}_(0.4){\sim}&\ow \G (Y\hker X)\ar [d]_{\ow \G \iota _{g}}\\
\ow \G X \ar [r]^=&\ow \G X&X\ar[l]_{u_{X}}^\sim\ar [r]^=&X\ar [r]^(0.4){u_{X}}_(0.4){\sim}&\ow \G X,}
\]
where the elementary equivalence of the first column follows from Lemma \ref{lem:amusing}(1) applied to $\G g$, that of the second column from Lemma \ref{lem:amusing}(2) applied to $g$, that of the third column from Lemma \ref{lem:simpl-good} and that of the fourth column from Lemma \ref{lem:amusing}(2) applied to $\iota_{g}$.
We can therefore conclude that $\G\iota_{g}$ is h-normal, with associated normality structure $\overline\W\G g: \overline\W\G X\to \overline\W\G Y$.
\end{proof}

Not only do we obtain an h-normal morphism upon looping homotopy fiber inclusions, but every h-normal morphism of reduced simplicial groups is weakly equivalent to a morphism of this type.  This is analogous to the fact that the normal subgroups of a fixed group $G$ are exactly the kernels of surjective homomorphisms with domain $G$.

\begin{prop}\label{prop:char-h-normal} If $f:G\to G'$ is an h-normal morphism of simplicial groups, then there is a simplicial map $g:X\to Y$ such that $f$ is weakly equivalent to $\G \iota_{g}:\G(Y\hker X)\to \G X$.
\end{prop}

\begin{proof}  Let $g:\overline \W G'\to Y$ be the normality structure associated to $f$, i.e., there is a commuting diagram of simplicial morphisms
\[
\xymatrix{\G X\ar [d]^{\G g}&\bullet \ar[d]\ar [l]^\sim \ar [r]_{\sim}&\cdots&\bullet \ar[d]\ar [l]^\sim\ar [r]_{\sim}&G'\ar [d]_{\pi_{f}} \\
\G Y\ar [d]^{\del_{g}}&\bullet\ar[d] \ar [l]^\sim \ar [r]_{\sim}&\cdots&\bullet \ar[d]\ar [l]^\sim\ar [r]_{\sim}&G'\hquot G\ar [d]_{\delta_{f}}\\
Y\hker X\ar [d]^{\iota _{g}}&\bullet\ar[d]\ar [l]^\sim\ar [r]_{\sim}&\cdots&\bullet\ar[d] \ar [l]^\sim\ar [r]_{\sim}&\ow G\ar [d]_{\ow f}\\
X&\bullet \ar [l]^\sim \ar [r]_{\sim}&\cdots&\bullet \ar [l]^\sim\ar [r]_{\sim}&\ow G'.}
\]
Applying $\G$ to the lower two rows of the diagram, we see that $\G \iota _{g}$ is equivalent to $\G \ow f$, which is in turn equivalent to $f$.
\end{proof}

H-conormality is banal in the simplicial context.

\begin{prop}\label{prop:allsimpl}  All simplicial maps are h-conormal.
\end{prop}

\begin{proof} Let $g:X\to Y$ be any simplicial map.  Lemmas \ref{lem:rigidconormal} and  \ref{lem:simpl-good}  together imply that $g$ is h-conormal.
\end{proof}

\begin{rmk}  Our analysis of h-normality and h-conormality in the simplicial setting is highly dependent on the fact that the monoidal product of simplicial sets is exactly the categorical product, so that every morphism of simplicial sets can be viewed as a morphism of comonoids.  In particular, every simplicial homomorphism of simplicial groups underlies a morphism of bimonoids.

We expect that results similar to Propositions \ref{prop:looping-conormal} and \ref{prop:char-h-normal} hold in a more general monoidal category as well, if we restrict to morphisms of bimonoids.  In the next section, we study h-normality in a monoidal category where the monoidal product is not the categorical product.
\end{rmk}

\subsubsection{The chain complex case}

As before, let $\cat {Alg}_{\Bbbk}$ denote the category of connected, augmented chain algebras, and let $\cat {Coalg}_{\Bbbk}$ denote the category of $1$-con\-nec\-ted, coaugmented chain coalgebras over a commutative ring $\Bbbk$, whose underlying chain complexes are degreewise projective and finitely generated.
We begin by considering the most ``extreme'' possible cases of h-normal and h-conormal morphisms in this framework.

\begin{ex}\label{ex:chcx} \begin{enumerate}
\item Let $A\in \ob \cat {Alg}_{\Bbbk}$.  Its unit map $\eta:\Bbbk \to A$ is always h-normal, with associated normality structure equal to the coaugmentation  $\Bbbk\to \Bar A$.  The identity map $A\xrightarrow = A$ is also h-normal, with associated normality structure equal to the counit $\Bar A \to \Bbbk$.

On the other hand, the augmentation $A\to \Bbbk$ is h-normal if and only if there is a coalgebra $C$ such that $\Om C \sim \Bar A$, as chain complexes.  In particular, for any $1$-reduced simplicial set $X$, the augmentation map $C_{*}\G^2X \to \Bbbk$ is h-normal, since  $\Om C_{*}X\sim C_{*}\G X\sim \Bar C_{*}\G ^2 X$ (cf. e.g., \cite{hpst}).

\item Let $C\in \cat {Coalg}_{\Bbbk}$. Arguments dual to those above show that its counit and identity map are h-conormal, while the coaugmentation map is  h-conormal if there is an algebra $A$ such that $\Bar A \sim \Om C$, as chain complexes.  In particular, for any $1$-reduced simplicial set $X$, the coaugmentation map $\Bbbk\to C_{*}X$ is h-conormal, since $\Bar C_{*}\G ^2 X\sim \Om C_{*}(X)$, as above.
\end{enumerate}
\end{ex}

We next consider a purely algebraic class of examples, which can be thought of as the homotopical, chain analogue of the fact that every subgroup of an abelian group is normal.

\begin{prop}\label{prop:abelian}  If $A$ and $A'$ are connected, commutative chain algebras, then any algebra morphism $f: A\to A'$ is  h-normal.  Dually, if $C$ and $C'$ are $1$-connected, cocommutative chain coalgebras, then any coalgebra morphism $g:C'\to C$ is h-conormal.
\end{prop}

\begin{proof}  It is a classical result that if $A$ is commutative, then $\Bar A$ admits a commutative multiplication
\[
\Bar A\otimes \Bar A \xrightarrow \nabla \Bar (A\otimes A) \xrightarrow {\Bar \mu} \Bar A,
\]
where $\nabla$ is the shuffle (Eilenberg-Zilber) equivalence, and $\mu$ is the multiplication map of $A$, which is an algebra map since it is commutative.  A simple calculation then shows that
\[
(\Bar A\ot_{ft_{\Bar}}A')\ot (\Bar A\ot_{ft_{\Bar}}A')\to \Bar A\ot_{ft_{\Bar}}A': (w\otimes a)\otimes (w'\otimes a')\mapsto w\cdot w'\otimes aa',
\]
where $\cdot$ denotes the multiplication on $\Bar A$ defined above, is a chain map, as well as associative and unital.  It is obvious that $\pi_{f}:A'\to A'\hquot A=\Bar A\ot_{ft_{\Bar}}A'$ is a morphism of chain algebras with respect to this multiplication. Moreover, the multiplicative structure on $A'\hquot A$ is such that
\[
(A'\hquot A)\hquot A'=\U A'\ot _{A'}(\U A \ot _{A} A')\cong \U A'\ot \U A\ot_{A}\Bbbk \cong \U A'\otimes \Bar A,\]
whence the obvious projection map $(A'\hquot A)\hquot A'\to \Bar A$ is a weak equivalence, since $\U A'$ is acyclic.  Lemma \ref{lem:rigidnormal} therefore implies that $f$ is h-normal.

Dualizing the proof in the algebra case, we obtain the desired result for chain coalgebras as well.
\end{proof}

For algebraic topologists, the most interesting source of examples of h-conormal and h-normal maps of chain complexes may be the following.

\begin{prop}\label{prop:ch-conormal-induced} If $g:X\to Y$ is a simplicial map, where $Y$ is $1$-reduced and both $X$ and $Y$ are of finite type, then $C_{*}g:C_{*}X\to C_{*}Y$ is an h-conormal map of chain coalgebras, with associated conormality structure $\Om C_{*}\iota _{g}:\Om C_{*}(Y\hker X)\to \Om C_{*}X$, i.e, $(\Om C_{*}\iota _{g}, C_{*}g)$ is a normal pair.
\end{prop}

It follows almost immediately from the proposition above that applying the normalized chains functor to an h-normal map of simplicial groups gives rise to an h-normal map of chain algebras, i.e., $C_{*}$ preserves both h-normality and h-conormality, under mild connectivity hypotheses, the proof is given below.

\begin{cor} \label{cor:ch-normal-induced}  If $f:G\to G'$ is an h-normal morphism of reduced simplicial groups, then $C_{*}f:C_{*}G\to C_{*}G'$ is an h-normal morphism of chain algebras.
\end{cor}

\begin{proof} It follows from the proof of Proposition \ref{prop:char-h-normal} that we may assume that $f=\G \iota _{g}:\G(Y\hker X) \to \G X$ for some morphism $g:X\to Y$ of $1$-reduced simplicial sets. Since $C_{*}\G\iota _{g}$ is weakly equivalent to $\Om C_{*}\iota_{g}$, Proposition \ref {prop:ch-conormal-induced} and Remark \ref{rmks:norm}(2) together imply that $C_{*}\G\iota _{g}$ is h-normal.
\end{proof}

The key to  the proof of  Proposition \ref{prop:ch-conormal-induced} is the following immediate consequence of Theorems 3.15 and 3.16 in \cite{hps-cohoch}, which gives an explicit, natural chain-level model for any simplicial principal fibration, incorporating multiplicative and comultiplicative structures.  Functoriality of h-conormality should hold for functors between twisted homotopical categories for which an analogous theorem is true.

\begin{thm}\label{thm:cohoch}   If $g:X\to Y$ is a simplicial map, where $Y$ is $1$-reduced and both $X$ and $Y$ are of finite type, then there is a natural weak equivalence of mixed bundles of  chain complexes
\[
\xymatrix{\Om C_{*}Y \ar[d]_{\sim}^{Sz_{Y}}\ar [r]^(0.4){\del_{C_{*}g}}&C_{*}Y\hker C_{*}X \ar [d]_{\sim}^{Sz_{g}}\ar [r]^(0.6){\iota_{C_{*}g}}&C_{*}X\ar [d]_{=}\\
C_{*}\G Y\ar [r]^(0.4){C_{*}\del_{g}}& C_{*}(Y\hker X)\ar [r]^(0.6){C_{*}\iota _{g}}&C_{*}X.}
\]
\end{thm}

\begin{rmk} In the diagram above, the map $Sz_{Y}:\Om C_{*}Y\xrightarrow \sim C_{*}\G Y$ is the natural chain algebra quasi-isomorphism first defined by Szczarba \cite{szczarba} (cf. Example \ref{ex:szczarba}), and $Sz_{g}$ is a natural extension of $Sz_{Y}$.
\end{rmk}

\begin{proof}[Proof of Proposition \ref{prop:ch-conormal-induced}]   There is an equivalence of extended bundles
{\small \[
\xymatrix{\Om C_{*}X \ar [d]_{\Om C_{*}g}\ar [r]^{Sz_{X}}_{\sim}&C_{*}\G X\ar [d]_{C_{*}\G g}\ar [r]^=&C_{*}\G X\ar [d]_{C_{*}\del_{\iota_{g}}}&\Om C_{*}X\ar [l]_{Sz_{X}}^\sim \ar [d]_{\del_{C_{*}\iota_{g}}}\ar [r]^=& \Om C_{*}X\ar [d]_{\pi_{\Om C_{*}\iota_{g}}}\\
\Om C_{*}Y\ar[d]_{\del_{C_{*}g}}\ar[r]^{Sz_{Y}}_{\sim}&C_{*}\G Y\ar [d]_{C_{*}\del_{g}}\ar[r]^(0.35){C_{*}\tilde\iota}_(0.35){\sim}&C_{*}\big(X\hker(Y\hker X)\big)\ar [d]_{C_{*}\iota_{\iota_{g}}}&C_{*}X\hker C_{*}(Y\hker X)\ar[l]_{Sz_{\iota_{g}}}^\sim\ar[d]_{\iota_{C_{*}\iota_{g}}}\ar [r]_(0.45){\sim}&\Om C_{*}X\hquot \Om C_{*}(Y\hker X)\ar [d]_{\delta_{\Om C_{*}\iota_{g}}}\\
C_{*}Y\hker C_{*}X\ar [d]_{\iota_{C_{*}g}}\ar [r]_{\sim}^{Sz_{g}}&C_{*}(Y\hker X)\ar[d]_{C_{*}\iota _{g}}\ar[r]^=&C_{*}(Y\hker X)\ar[d]_{C_{*}\iota g}& C_{*}(Y\hker X)\ar[l]_{=}\ar[d]_{C_{*}\iota _{g}}\ar[r]^{u}_{\sim}&\Bar\Om C_{*}(Y\hker X)\ar[d]_{\Bar \Om C_{*}\iota _{g}}\\
C_{*}X\ar [r]^=&C_{*}X\ar [r]^=&C_{*}X&C_{*}X\ar [l]_=\ar[r]^{u}_{\sim}&\Bar \Om C_{*}X,}
\] }
where the elementary equivalence of the first column arises from applying Theorem \ref{thm:cohoch} to $g$, that of the second column from applying Lemma \ref{lem:simpl-good} to $g$, that of the third column from applying Theorem \ref{thm:cohoch} to $\iota_{g}:Y\hker X \to X$ and that of the fourth column from applying Lemma \ref{lem:amusing}(2) to $C_{*}\iota _{g}$.   We can therefore conclude that $C_{*}g$ is h-conormal, with associated conormality structure $\Om C_{*}\iota_{g}$.
\end{proof}

\appendix

\section{Twisting functions and twisting cochains}

\subsection{Twisting functions}\label{sec:twistingfcns}
We recall here the simplicial constructions that play an crucial role in this article.

\begin{defn} Let $X$ be a simplicial set and $G$ a simplicial group,
where the neutral element in any dimension is noted $e$. A
degree $-1$ map of graded sets $\tau :X\to G$ is a \emph{twisting
function} if
\begin{align*}
d _{0}\tau (x)&=\big(\tau (d_{0}x)\big)^{-1}\tau (d
_{1}x)\\
d _{i}\tau (x)&=\tau (d
_{i+1}x)\quad i>0\\
s_{i}\tau (x)&=\tau (s_{i+1}x)\quad i\geq 0\\
\tau (s_{0}x)&=e
\end{align*}
for all $x\in X$.
\end{defn}

\begin{rmk}\label{rmk:twisting-fcn} Let $X$ be a reduced simplicial set, and let $\G X$ denote the Kan simplicial loop group on $X$ \cite{may}.  Let $\bar x\in (\G X)_{n-1}$ denote a free group generator, corresponding to $x\in X_{n}$. There is a universal, canonical twisting function $\tau_{X}:X\to \G X$, given by $\tau_{X}(x)=\bar x$.  This twisting function is universal in the sense that it mediates a bijection between the set of twisting functions with source $X$ and the set of morphisms of simplicial groups with source $\G X$.
\end{rmk}

\begin{rmk}  If $f:X\to Y$ is a simplicial map, $\tau: Y\to G$ is a twisting function, and $\vp: G\to H$ is a simplicial homomorphism, then $\vp t f:X\to H$ is clearly also a twisting function.
\end{rmk}

\begin{defn} Let $\tau :X\to G$ be a twisting function, where $G$ operates on the left on a simplicial set $Y$.
The \emph{twisted cartesian product} of $X$
and $Y$, denoted $X\times _{\tau}Y$, is a simplicial set such
that
$(X\times_{\tau} Y)_{n}=X_{n}\times Y_{n}$, with faces and
degeneracies given by
\begin{align*}
d_{0}(x,y)&=(d _{0}x,\tau (x)\cdot d _{0}y)\\
d _{i}(x,y)&=(d _{i}x,d_{i}y)\quad i>0\\
s_{i}(x,y)&=(s _{i}x,s_{i}y)\quad i\geq 0.
\end{align*}
\end{defn}

If $Y$ is a Kan complex, then the projection $X\times_{\tau} Y\to X$ is
a Kan fibration \cite {may}.

\begin{defn} Let $G$ be a simplicial group, where the neutral element in each dimension is denoted $e$.  The \emph{Kan classifying space} of $G$ is the simplicial set $\overline \W G$ such that $\overline \W G_{0}=\{(\;)\}$ and for all $n>0$,
\[
\overline \W G_{n}=G_{0}\times \cdots \times G_{n-1},
\]
with face maps given by
\[
d_{i}(a_{0},...,a_{n-1})=\begin{cases} (d_{1}a_{1},...,d_{n-1}a_{n-1})&: i=0\\ (a_{0},...,a_{i-2}, a_{i-1}\cdot d_{0}a_{i}, d_{1}a_{i+1},..., d_{n-1-i}a_{n-1})&: 0<i<n\\ (a_{0},...,a_{n-1})&: i=n\end{cases}
\]
and degeneracies given by $s_{0}\big ( (\;)\big)=(e)$, while
\[
s_{i}(a_{0},...,a_{n-1})=\begin{cases} (e, s_{0} a_{0},...,s_{n-1}a_{n-1})&:i=0\\ (a_{0},...,a_{i-1}, e,s_{0}a_{i},...,s_{n-1-i}a_{n-1})&: 0<i<n\\(a_{0},...,a_{n-1},e)&: i=n.\end{cases}
\]
\end{defn}

\begin{rmk} Our definition of the Kan classifying space differs for the sake of convenience from that in \cite{may} by a permutation of the factors.
\end{rmk}

\begin{rmk}\label{rmk:couniv-twist} The classifying space $\overline \W G$ deserves its name, as homotopy classes of simplicial maps into $\overline \W G$ classify twisted cartesian products with fiber $G$.  The universal $G$-bundle is a twisted cartesian product
\[
G\hookrightarrow \overline \W G\times_{\nu_{G}}G \twoheadrightarrow \overline \W G,
\]
with contractible total space, where $\nu_{G}: \overline \W  G\to G$ is the natural (couniversal) twisting function defined by
\[
\nu_{G} (a_{0},...a_{n-1})=a_{n-1},
\]
and $G$ acts on itself by left multiplication.  The twisting function $\nu_{G}$ is couniversal in the sense that it mediates a bijection between the set of  twisting functions with target $G$ and the set of simplicial maps with target $\overline \W G$.
\end{rmk}

\begin{rmk}\label{rmk:eta-equiv} The classifying space functor $\overline \W $ is right adjoint to the Kan loop group functor $G$. Furthermore, the unit map $\eta_{X}: X\to \overline \W \G X$ is a weak equivalence for all reduced simplicial sets $X$.
\end{rmk}

\begin{rmk}\label{rmk:unit-counit}  It follows from the universality of $\tau_{X}:X\to \G X$ and the couniversality of $\nu_{G}:\overline \W G \to G$ that the diagram
\[
\xymatrix{\overline\W G\ar [d]_{\eta_{G}}\ar[r]^{\nu_{G}}\ar[dr]^{\tau_{\overline \W G}}&G\\
\overline\W \G\overline \W G\ar [r]^{\nu_{\G\overline \W G}}&\G\overline \W G\ar[u]_{\ve_{G}}      }
\]
commutes for all simplicial groups $G$.
\end{rmk}

\subsection{Twisting cochains}\label{sec:twistingcochains}

\begin{notn} In this section we apply the following notational conventions.
\begin{itemize}
\item We apply the Koszul sign convention for commuting elements  of a graded module or for commuting a morphism of graded modules past an element of the source module.  For example,  if $V$ and $W$ are graded algebras and $v\otimes w, v'\otimes w'\in V\otimes W$, then 
\[
(v\otimes w)\cdot (v'\otimes w')=(-1)^{|w|\cdot |v'|}vv'\otimes ww'.
\] 
Furthermore, if $f:V\to V'$ and $g:W\to W'$ are morphisms of graded modules, then for all $v\otimes w\in V\otimes W$,
\[
(f\otimes g)(v\otimes w)=(-1)^{|g|\cdot |v|} f(v)\otimes g(w).
\]
\item The \emph {suspension} endofunctor $s$ on the category of graded modules over a commutative ring $\Bbbk$ is defined on objects $V=\bigoplus _{i\in \mathbb Z} V_ i$ by
$(sV)_ i \cong V_ {i-1}$.  Given a homogeneous element $v$ in
$V$, we write $sv$ for the corresponding element of $sV$. The suspension $s$ admits an obvious inverse, which we denote $\si$.
\item Let $T$ denote the endofunctor on the category of free graded $\Bbbk$-modules, for a commutative ring $\Bbbk$, given by
\[
TV=\oplus _{n\geq 0}V^{\otimes n},
\]
where $V^{\otimes 0}=R$.  An elementary tensor belonging to the summand $V^{\otimes n}$ of $TV$ is denoted $v_{1}|\cdots |v_{n}$, where $v_{i}\in V$ for all $i$.
\end{itemize}
\end{notn}

We begin by recalling the cobar and bar constructions in the differential graded framework.
Let $\cat{Coalg}_{\Bbbk}$ denote the category of $1$-connected, coaugmented chain coalgebras over a commutative ring $\Bbbk$, i.e., of coaugmented comonoids in $\cat {Ch}_{\Bbbk}$ such that $C_{<0}=0$, $C_{0}=\Bbbk$, and $C_{1}=0$. Let $\cat {Alg}_{\Bbbk}$ denote the category of connected, augmented chain algebras over $\Bbbk$, i.e., of augmented monoids $B$ in $\cat {Ch}_{\Bbbk}$ such that $B_{<0}=0$ and $B_{0}=R$.

 The \emph{cobar construction} functor $\Om:\cat {Coalg}_{\Bbbk}\to \cat {Ch}_{\Bbbk}$, defined by
\[
\Om C= \left(T (\si C_{>0}), d_{\Om}\right)
\]
where, if $d$ denotes the differential on $C$, then
\begin{align*}
d_{\Om}(\si c_{1}|\cdots|\si c_{n})=&\sum _{1\leq j\leq n}\pm \si c_{1}|\cdots |\si (dc_{j})|\cdots |\si c_{n}\\
&+\sum _{1\leq j\leq n}\pm \si c_{1}|...|\si c_{ji}|\si c_{j}{}^{i}|\cdots |\si c_{n},
\end{align*}
with signs determined by the Koszul rule, where the reduced comultiplication applied to $c_{j}$ is $c_{ji}\otimes c_{j}{}^{i}$ (using Einstein implicit summation notation).

Observe that the graded $\Bbbk$-vector space underlying $\Om C$ is naturally a free associative algebra, with multiplication given by concatenation. The differential $d_{\Om }$ is a derivation with respect to this concatenation product, so that $\Om C$ is itself a chain algebra.  We can and do therefore choose to consider the cobar construction to be a functor $\Om :  \cat {Coalg}_{\Bbbk}\to \cat {Alg}_{\Bbbk}$.  Moreover, if the chain complex underlying a chain coalgebra is degreewise finitely generated and projective, then the same is true of its cobar construction.

 The \emph{bar construction} functor from $\cat{Alg}_{\Bbbk}$ to $\cat {Ch}_{\Bbbk}$, defined by
\[
\Bar B=\left(T (sB_{>0}), d_{\Bar}\right)
\]
where, if $d$ is the differential on $B$, then (modulo signs, which are given by the Koszul rule)
\begin{align*}
d_{\Bar}(sb_{1}|\cdots|sb_{n})=&\sum _{1\leq j\leq n}\pm sb_{1}|\cdots |s(db_{j})|\cdots |sb_{n}\\
&+\sum _{1\leq j<n}\pm sb_{1}|...|s(b_{j}b_{j+1})|\cdots |sb_{n}.
\end{align*}

Observe that the graded $\Bbbk$-module underlying $\Bar B$ is naturally a cofree coassociative coalgebra, with comultiplication given by splitting of words. The differential $d_{\Bar}$ is a coderivation with respect to this splitting comultiplication, so that $\Bar B$ is itself a chain coalgebra. We can and do therefore choose to consider the bar construction to be a functor $\Bar :  \cat {Alg}_{\Bbbk}\to \cat {Coalg}_{\Bbbk}$. Moreover, if the chain complex underlying a chain algebra is degreewise finitely generated and projective, then the same is true of its bar construction.

 Let $\eta: Id\to \Bar \Om$ denote the unit of this adjunction.  It is well known that for all $1$-connected, coaugmented chain coalgebras $C$, the counit map
\begin{equation}\label{eqn:unit-barcobar}
\eta_{C}:C\xrightarrow\simeq\Bar\Om C
\end{equation}
is a quasi-isomorphism of chain coalgebras.

\begin{defn}
A \emph{twisting cochain} from a $1$-connected, coaugmented chain coalgebra $(C,d)$ with comultiplication $\Delta$ to a connected, augmented chain algebra $(A,d)$ with multiplication $m$ consists of a linear map $t:C\to A$ of degree $-1$ such that
\[dt+td=m (t\otimes t)\Delta.\]
\end{defn}

\begin{rmk}
A twisting cochain $t:C\to A$ induces both a chain algebra map
\[
\alpha _{t}:\Om C\to A
\]
specified by $\alpha _{t}(\si c)=t(c)$ and a chain coalgebra map (the adjoint of $\alpha_{t}$ under the $(\Om, \Bar)$-adjunction)
\[
\beta _{t}:C\to \Bar A, 
\]
 satisfying
 \[
\alpha_{t}=\ve_{A}\circ \Om \beta_{t}\quad\text{and}\quad \beta _{t}=\Bar\alpha_{t}\circ \eta_{C}.
\]
It follows that $\alpha_{t}$ is a quasi-isomorphism if and only if $\beta_{t}$ is a quasi-isomorphism.
\end{rmk}

 \begin{ex} Let $C$ be a $1$-connected, coaugmented chain coalgebra. The \emph{universal  twisting cochain}
\[
t_{\Om}:C\to \Om C
\]
is defined by $t_{\Om }(c)=s^{-1} c$ for all $c\in C$, where $\si c$ is defined to be $0$ if $|c|=0$.  Note that  $\alpha_{t_{\Om}}=Id_{\Om C}$, so that $\beta _{t_{\Om}}=\eta_{C}$.  Moreover, $t_{\Om}$ truly is universal, as all twisting cochains $t:C\to A$ factor through $t_{\Om}$, since the diagram
 \[
\xymatrix{C\ar[r] ^{t_{\Om}}\ar [dr]_{t}&\Om C\ar [d]^{\alpha_{t}}\\ &A}
\]
 always commutes.
 \end{ex}

  \begin{ex} Let $A$ be a connected, augmented chain algebra. The \emph{couniversal  twisting cochain}
\[
t_{\Bar}:\Bar A\to A
\]
is defined by $t_{\Bar  }(sa)=a$ for all $a\in A$ and $t_{\Bar}(sa_{1}|\cdots |sa_{n})=0$ for all $n>1$.  Note that  $\beta_{t_{\Bar}}=Id_{\Bar A}$, so that $\alpha _{t_{\Om}}=\ve_{A}$.  Moreover, $t_{\Bar}$ truly is couniversal, as all twisting cochains $t:C\to A$ factor through $t_{\Bar}$, since the diagram
 \[
\xymatrix{&\Bar A\ar [d]^{t_{\Bar}}\\ C\ar[ur]^{\beta_{t}}\ar [r]_{t}&A}
\]
 always commutes.
 \end{ex}

 \begin{ex} \label{ex:szczarba} Let $K$ be a reduced simplicial set, and let $\G K$ denote its Kan loop group. In 1961 \cite {szczarba}, Szczarba gave an explicit formula for a twisting cochain
\[
sz_{K}:C_{*}K\to C_{*}\G K,
\]
natural in $K$.  The associated chain algebra map
\begin{equation}\label{eqn:szczarba}
Sz _{K}:=\alpha_{sz_{K}}:\Om C_{*}K\to C_{*}\G K
\end{equation}
is a quasi-isomorphism of chain algebras for every $1$-reduced $K$ \cite{hess-tonks}.  It follows that the induced coalgebra map
\[
Sz_{K}^\sharp:=\beta_{sz_{K}}:C_{*}K\to \Bar C_{*}\G K
\]
is also a quasi-isomorphism when $K$ is $1$-reduced.
\end{ex}

\begin{rmk}  If $t:C\to A$ is a twisting cochain, $f:C'\to C$ is a chain coalgebra map and $g:A\to A'$ is a chain algebra map, then $gtf:C'\to A'$ is also a twisting cochain.
\end{rmk}

\begin {defn} Let $t:C\to A$ be a twisting cochain.  Let $M$ be a right $A$-module,
where $\rho :M\otimes A\to M$ is the $A$-action, and let  $N$ be a left $C$-comodule, where $\lambda
:N\to C\otimes N$ is the $C$-coaction.  Let $d$ denote the differential on both $M$ and $N$. The \emph{twisted tensor product} of $M$ and $N$ over $t$ is a chain complex  $M\ot_tN=(M\otimes N,
D_{t})$, where 
\[
D_{t}=d\otimes N+M\otimes d+
(\rho\otimes N)(M\otimes t\otimes N)(M\otimes
\lambda).
\]
\end{defn}

\begin{rmk} There are analogous constructions for left $A$-modules and
right $C$-comodules, using the twisting isomorphism $A\otimes C\cong
C\otimes A$.
\end{rmk}

\section{Twisting structures}\label{sec:twisting}

We recall in this section  a categorical structure that conveniently generalizes both  twisting cochains from differential graded coalgebras to differential graded algebras and twisting functions from simplicial sets to simplicial groups.  Such a definition was first formulated in \cite{hess-lack}; this presentation is a variant of the somewhat less highly category-theoretical formulation in \cite{hess-scott}, strongly emphasizing a bundle-theoretic perspective.
Twisting structures provide a useful formalism for taking into account important monoidal structures that are not
given by the categorical product.

We begin by studying mixed bundles, which are the building blocks of a twisting structure.  We then define twisting structures and study Nomura-Puppe sequences in categories endowed with a twisting structure.  We conclude with two explicit examples, based on the categories of simplicial sets and of chain complexes.

Throughout this section $(\cat M, \otimes, I)$ denotes a monoidal category.   If $A$ and $A'$ are  monoids in $\cat M$, then ${}_{A}\cat {Mod}_{A'}$ denotes the category of $(A,A')$-bimodules.  Similarly, ${}_{C'}\cat {Comod}_{C}$ denotes the category of $(C',C)$-bicomodules, where $C'$ and $C$ are comonoids.  Let $\cat {Mon}$ and $\cat {Comon}$ denote the categories of augmented monoids and coaugmented comonoids in $\cat M$, respectively.

\subsection{Mixed bundles}

The following definitions are classical.

\begin{defn}  Let $A$ be a monoid in $\cat M$.  Let $(M,\rho)$ be a right $A$-module, and let $(N,\lambda)$ be a left $A$-module.  The \emph{tensor product of $M$ and $N$ over $A$} is the coequalizer in $\cat M$
\[
M\ot_A  N:=\operatorname{coequal}(M\otimes A\otimes N \egal{\rho\otimes N}{M\otimes \lambda} M\otimes N),
\]
if it exists.

 Let $C$ be a comonoid in $\cat M$.  Let $(M,\rho)$ be a right $C$-comodule, and let $(N,\lambda)$ be a left $C$-comodule.  The \emph{cotensor product of $M$ and $N$ over $C$} is the equalizer in $\cat M$
 \[
M\square_C N:=\operatorname{equal}(M\otimes N\egal{\rho\otimes N}{M\otimes \lambda} M\otimes C\otimes N),
\]
if it exists.
 \end{defn}

Objects of $\cat M$ that are endowed with a compatible action and coaction are our primary object of study here.

\begin{defn} If $C$ is a comonoid in $\cat M$, and $A$ is a monoid in $\cat M$, then ${}_{C}\cat {Mix}_{A}$ is the category of  \emph{$(C,A)$-mixed modules}, where
{\small
\begin{multline*}
\ob {}_{C}\cat {Mix}_{A}\\
=\{ (M, \lambda, \rho)\mid M\in \cat M, (M,\lambda)\in {}_{C}\cat {Comod}, (M,\rho)\in \cat {Mod}_{A}, (C\otimes \rho) (\lambda \otimes A)=\lambda\rho\},
\end{multline*}}
and morphisms preserve both the action and the coaction.
\end{defn}

\begin{defn} Let $(\cat M, \otimes, I)$ be a monoidal category.  The \emph{category of mixed bundles in $\cat M$}, denoted $\cat {Bun}$, has as objects sequences of morphisms
\[
A\xrightarrow i N\xrightarrow p C,
\]
where $A\in \ob \cat {Mon}$, $C\in \ob \cat {Comon}$, $N\in {}_{C}\cat {Mix}_{A}$, $i$ is a morphism of right $A$-modules, and $p$ is a morphism of left $C$-comodules.  Morphisms in $\cat {Bun}$ are the obvious structure-preserving triples of morphisms in $\cat M$.

The \emph {source functor} $\mathfrak S:\cat {Bun}\to \cat {Mon}$ and \emph{target functor} $\mathfrak T:\cat {Bun}\to \cat {Comon}$ are specified by
\[
\mathfrak S (A\xrightarrow i N\xrightarrow p C)=A \quad\text{and}\quad \mathfrak T (A\xrightarrow i N\xrightarrow p C)=C.
\]
\end{defn}

Twisting structures are defined for monoidal categories $(M, \otimes, I)$ that admit a particularly nice class of mixed modules, as in the definition below.  We use in this definition the set-theoretic hierarchy of categories established in Appendix A.2 of \cite{michel-thesis}.

 \begin{defn}\label{defn:tt}  Let $(\cat M, \otimes, I)$ be a monoidal category.  A \emph{twistable triple} 
 \[
 (\cat {mon}, \cat{comon}, \operatorname{mix})
 \]
 consists of full subcategories $\cat {mon}$ and $\cat{comon}$ of $\cat {Mon}$ and $\cat {Comon}$, respectively, and a pseudofunctor 
 \[
\operatorname{mix}: \cat{comon}^{}\times \cat{mon}^{op}\to \cat {CAT}: (C,A)\mapsto {}_{C}\cat {mix}_{A}
 \]
 where $\cat {CAT}$ is the $2$-XL-category of all (ordinary) categories, such that for all $A,A'\in \ob \cat{mon}$ and $C,C'\in \ob\cat {comon}$,
 \begin{enumerate}
 \item ${}_{C}\cat {mix}_{A}$ is a full subcategory of  $ {}_{C}\cat {Mix}_{A}$;
 \item cotensoring gives rise to a functor
 \[
 C'_{g}\square_{C}-:{}_{C}\cat {mix}_{A} \to {}_{C'}\cat {mix}_{A}
 \]
 for all $g\in \cat{comon}(C',C)$, where $C'_{g}$ denotes $C'$ endowed with the right $C$-comodule structure induced by $g$ and its usual left $C'$-comodule structure such that iterated cotensoring is associative up to natural isomorphism;
 \item tensoring gives rise to a functor
 \[
 -\otimes_{A}{}_{f}A':{}_{C}\cat {mix}_{A}\to {}_{C}\cat {mix}_{A'}
 \]
 for all $f\in \cat {mon}(A,A')$, where ${}_{f}A'$ denotes $A'$ endowed with the left $A$-module structure induced by $f$ and its usual right $A'$-module structure such that iterated tensoring is associative up to natural isomorphism; and
 \item the natural map
 \[
 (C'_{f}\square_{C}M)\otimes _{A}{}_{g}A'\to C'_{f}\square_{C}(M\otimes _{A}{}_{g}A')
 \]
 is an isomorphism for all $f\in \cat{comon}(C',C)$, $M\in {}_{C}\cat {mix}_{A}$ and $g\in \cat{mon}(A,A')$.
 \end{enumerate}
 If $(\cat {mon}, \cat {comon}, \operatorname{mix})$ is a twistable triple, let $\cat {bun}$ denote the full subcategory of $\cat {Bun}$ of which the objects $A\xrightarrow j M \xrightarrow p C$ are  such that $A\in \ob \cat {mon}$,  $C\in \ob\cat {comon}$ and $M\in \ob {}_{C}\cat {mix}_{A}$.
 \end{defn}

\begin{ex} Consider the monoidal category $\cat {sSet}$ of simplicial sets. The monoidal product is the categorical product, which implies that all objects are naturally comonoids, where the comultiplication is the diagonal map, and that an $X$-comodule structure on $Y$ is equivalent to the existence of a morphism $Y\to X$.  Cotensor products are thus simply pullbacks.  

Let $\cat {sSet}_{0}$  and $\cat {sGr}$ denote the categories of reduced simplicial sets and of simplicial groups, respectively, and let
 \[
\operatorname{mix}: \cat{sSet_{0}}^{}\times \cat{sGr}^{op}\to \cat {CAT}: (X,G)\mapsto {}_{X}\cat {Mix}_{G},
 \]
i.e., we consider \emph{all} mixed modules. On morphisms $\operatorname{mix}$ is given by extension of coefficients in the comonoid variable and restriction of coefficients in the monoid variable.  It is easy to check that the triple $(\cat {sGr}, \cat {sSet}, \operatorname{mix})$ is twistable, since a simplicial morphism $p:Y\to X$ where $Y$ admits a right $G$-action corresponds to a mixed $(X,G)$-module if and only if $p(y\cdot a)=p(y)$ for all $y\in Y_{n}$, $a\in G_{n}$ and $n\geq 0$.
\end{ex}

\begin{ex}\label{ex:tt-ch} Let $\Bbbk$ be any commutative ring, and let $\cat {dgProj}_{\Bbbk}$ denote the category of differential graded $\Bbbk$-modules that are projective and finitely generated in each degree, where the monoidal product $\otimes$ is the usual graded tensor product over $\Bbbk$.  Let $\cat {mon}$ denote the category of augmented, connected chain algebras and $\cat {comon}$ the category of coaugmented, $1$-conneced chain coalgebras in $\cat {dgProj}_{\Bbbk}$.  
Let
 \[
\operatorname{mix}: \cat{comon}^{}\times \cat{mon}^{op}\to \cat {CAT}: (C,A)\mapsto {}_{C}\cat {mix}_{A},
 \]
be the pseudofunctor such that objects of ${}_{C}\cat {mix}_{A}$ are mixed modules $M$ with underlying (nondifferential) graded mixed module of the form $C\otimes X\otimes A$ for some graded $\Bbbk$-module $X$. On morphisms $\operatorname{mix}$ is given by extension of coefficients in the comonoid variable and restriction of coefficients in the monoid variable. 

For all $M\in \ob {}_{C}\cat {mix}_{A}$ with underlying graded $\Bbbk$-module $C\otimes X\otimes A$ and any coalgebra morphism $g:C'\to C$, there is a natural isomorphism between the graded $\Bbbk$-module underlying $C'_{g}\square_{C}M$ and $C'\otimes X\otimes A$; the dual result holds for algebra morphisms.  It is therefore straightforward to show that the triple $(\cat{mon}, \cat {comon}, \operatorname{mix})$ is twistable.  
\end{ex}

Certain types of mixed bundles play a key role in the definition of twisting structures.

\begin{defn}  Let $(\cat M, \otimes, I)$ be a monoidal category, and let $(\cat {mon}, \cat {comon}, \operatorname{mix})$ be a twistable triple. Let $\zeta=(A\xrightarrow j N\xrightarrow q C)\in \ob \cat {bun}$.   Let $\eta:I\to C$ and $\ve: A\to I$ denote the coaugmentation of $C$ and the augmentation of $A$, respectively.

If $j= \eta\square_{C} N$, then $\zeta$ is \emph{principal}.  If $q=N\otimes_{A}\ve$, then $\zeta$ is \emph{coprincipal}.  A bundle that is both principal and coprincipal is called \emph{biprincipal}.

The full subcategory of $\cat {bun}$ consisting of biprincipal bundles is denoted $\cat {biprin}$.
\end{defn}

\begin{rmk}\label{rmk:biprin} Let $(\cat M, \otimes, I)$ be a monoidal category, and let $(\cat {mon}, \cat {comon}, \operatorname{mix})$ be a twistable triple. Let $\zeta=(A\xrightarrow j N\xrightarrow q C)\in \ob \cat{bun}$. If $\zeta$ is principal, then $A \cong I\square_{C} N$, while if $\zeta$ is coprincipal, then $C\cong N\otimes_{A}I$.  It follows that if $\zeta$ is principal, then $j$ is naturally a morphism in ${}_{C}\cat {mix}_{A}$, since $\eta:I\to C$ is a morphism of $C$-bicomodules, where $I$ is endowed with the trivial $C$-bicomodule structure induced by $\eta$.  Similarly, if $\zeta$ is coprincipal, then $q$ is naturally a morphism in ${}_{C}\cat {mix}_{A}$.  We conclude that if $\zeta$ is biprincipal, then it can be seen as a sequence of morphisms in ${}_{C}\cat {mix}_{A}$.
\end{rmk}

\begin{defn}  Let $(\cat M, \otimes, I)$ be a monoidal category, and let $(\cat {mon}, \cat {comon}, \operatorname{mix})$ be a twistable triple. Let $\zeta=(A\xrightarrow j M\xrightarrow q C)\in \ob\cat {bun}$. 

If $\zeta$ is coprincipal, and
$f:A\to A'$ is a morphism in $\cat {mon}$, then  the \emph{coprincipal bundle induced by $f$} is
\[
f_{*}(\zeta)=(A'\xrightarrow {j\ot_{A}A'} M\ot_{A} A' \xrightarrow{M\otimes_{A}\ve'} C),
\]
where $\ve':A'\to I$ is the augmentation of $A'$.

 If $\zeta $ is principal, and $g: C'\to C$ is a morphism in $\cat {comon}$, then  the \emph{principal bundle induced by $g$} is
\[
g^*(\zeta)=(A \xrightarrow {\eta'\square _{C}M} C'\square_{C}M \xrightarrow{g\square_{C}M} C'),
\]
where $\eta': I\to C'$ is the coaugmentation of $C'$.
\end{defn}

Note that condition (3) of Definition \ref{defn:tt} implies that $f_{*}(\zeta)$ is indeed coprincipal, while condition (2) of the same definition implies the dual result.  It follows then from condition (4) that if $\zeta$ is biprincipal, then so are $f_{*}(\zeta)$ and $g^*(\zeta)$. 

\begin{rmk}\label{rmk:commute}  Let $(\cat M, \otimes, I)$ be a monoidal category, and let $(\cat {mon}, \cat {comon}, \operatorname{mix})$ be a twistable triple. Let $\zeta=(A\xrightarrow j M\xrightarrow q C)\in \ob\cat {biprin}$.  If $f:A\to A'$ is a morphism in $\cat {mon}$, and $g:C'\to C$ is a morphism in $\cat {comon}$, then there are natural isomorphisms
\[
f_{*}g^*(\zeta)\cong g^*f_{*}(\zeta),
\]
by condition (4) of Definition \ref{defn:tt}.
\end{rmk}

\begin{rmk}\label{rmk:bndlmap} Let $(\cat M, \otimes, I)$ be a monoidal category, and let $(\cat {mon}, \cat {comon}, \operatorname{mix})$ be a twistable triple. Let $\zeta=(A\xrightarrow i M \xrightarrow p C)$ and $\xi=(B \xrightarrow j N \xrightarrow q D)$ be objects in $\cat {bun}$. Let $\vp=(\alpha, \gamma, \beta): \zeta \to \xi$ be a morphism of bundles.

If $\zeta$ and $\xi $ are coprincipal, and
\[
\xymatrix {A \ar[d]_{}^\alpha\ar [r]^f& A'\ar [d]_{}^{\alpha'}\\ B\ar [r]^{g}& B'}
\]
is a commuting diagram of monoid morphisms, then there is an induced morphism of bundles
\[
(\alpha, \alpha')_{*}: f_{*}(\zeta) \to g_{*}(\xi),
\]
given explicitly by
\[
\xymatrix{A'\ar [d]^{\alpha'}\ar[rr]^{i\ot_{A}A'}&&M\ot_{A}A'\ar[d]^{\gamma'}\ar [rr]^{M\ot_{A}\ve_{A'}}&&C\ar [d]^\beta\\ B'\ar[rr]^{j\ot_{B}B'}&&N\ot_{B}B'\ar [rr]^{N\ot_{B}\ve_{B'}}&&D,}
\]
where $\gamma'$ is the unique map induced on the coequalizers by $\gamma$, $\alpha$ and $\alpha'$.  In particular, since
\[
\xymatrix {A \ar@{=}[d]\ar @{=} [r]& A\ar [d]^f\\ A\ar [r]^{f}& A'}
\]
commutes for every monoid morphism $f:A\to A'$, there is a morphism of coprincipal bundles
\[
(A,f)_{*}:\zeta \to f_{*}(\zeta).
\]

Dually, if $\zeta$ and $\xi$ are principal, and
\[
\xymatrix{C' \ar[d]_{}^{\beta '}\ar [r]^f& C\ar [d]_{}^{\beta}\\ D'\ar [r]^{g}& D}
\]
is a commuting diagram of comonoid morphisms, then there is an induced morphism of bundles
\[
(\beta',\beta)_{*}:f^*(\zeta) \to g^*(\xi),
\]
given explicitly by
\[
\xymatrix {A \ar[d]^{\alpha}\ar [rr]^{\eta_{C'}\square_{C}M}&&C'\square_{C}M \ar[d]^{\gamma''}\ar [rr]^{f\square_{C}M}&&C' \ar[d]^{\beta'}\\
			B\ar [rr]^{\eta_{D'}\square_{D}N}&& D'\square_{D}N\ar [rr]^{g\square_{D}D'}&&D'}
\]
where $\gamma''$ is the unique map induced on the equalizers by $\gamma$, $\beta$ and $\beta'$.  In particular, since
\[
\xymatrix{C'\ar [d]^{g}\ar [r]^g&C\ar @{=}[d]\\ C\ar@{=}[r]&C}
\]
commutes for every comonoid morphism $g:C'\to C$, there is a morphism of principal bundles
\[
(g, C)_{*}:g^*(\xi)\to \xi.
\]
\end{rmk}

\subsection{Twisting structures: definition}

The following definition is a pragmatic and simplified variant of the definition of twisting structures formulated in \cite{hess-lack}.

 \begin{defn}\label{defn:twisting} A \emph{quasitwisting structure} on a  monoidal category $(\cat M, \otimes , I)$ endowed with a twistable triple $(\cat {mon}, \cat {comon}, \operatorname{mix})$ consists of a pair of functors
 
 \[
\zeta: \cat {mon} \to \cat {biprin}\quad\text{and}\quad \xi: \cat {comon} \to \cat {biprin}
\]
 such that $\mathfrak S\circ \zeta = Id$ and $\mathfrak T\circ \xi =Id$, while $\Om=\mathfrak S\circ \xi$ and $\Bar =\mathfrak T\circ \zeta $ form an adjoint pair 
 \[
\cat {comon}\adjunct{\Om}{\Bar}\cat {mon}.
\]
\end{defn}

\begin{notn} For all monoids $A$ and comonoids $C$ in $\cat N$, we write
\[
\zeta(A)= (A\xrightarrow {j_{A}} \U A \xrightarrow {q_{A}} \Bar A)
\]
and
\[
\xi(C)=(\Om C \xrightarrow {i_{C}} \P C \xrightarrow {p_{C}} C)
\]
and call these the \emph{classifying bundles} for $A$ and $C$, for reasons that should become more clear as we explain the role that these special bundles play.
\end{notn}

\begin{defn} \label{defn:twist} Let $u: Id\to \Bar \Om$ and $v: \Om \Bar \to Id$ denote the unit and counit of the $(\Om, \Bar)$-adjunction. A quasitwisting structure $(\Om, \Bar, \zeta, \xi)$ on a monoidal category $(\cat M, \otimes , I)$ endowed with a twistable triple $(\cat {mon}, \cat {comon}, \operatorname{mix})$ is a \emph{twisting structure} if
\begin{enumerate}
\item for every morphism $g:C\to \Bar A$ in $\cat {comon}$, there are natural isomorphisms of bundles
\[
\xymatrix{g^{*}\big(\zeta (A)\big) \ar[d]^\cong &=&\big(A\ar[d]^=\ar [r]&C \square_{\Bar A} \U A\ar[d]_{}^{\cong}\ar[r]&C\big)\ar[d]^=\\
(g^\flat )_{*}\big(\xi (C)\big)&=&\big(A\ar [r] &\P C\ot_{\Om C}A\ar [r] &C\big)},
\]
where $g^\flat:\Om C\to A$ denotes the transpose of $g$;
\item for every morphism $f:A\to A'$ in $\cat {mon}$, there is a natural isomorphism of bundles
\[
\xymatrix{(\Bar f)^{*}\big(\zeta (A')\big) \ar[d]^\cong &=&\big(A'\ar[d]^=\ar [r]&\Bar A\square_{\Bar A'}\U A'\ar[d]_{}^{\cong}\ar[r]&\Bar A\big)\ar[d]^=\\
f_{*}\big(\zeta (A)\big)&=&\big(A'\ar [r] &\U A\ot_{A}A'\ar [r] &\Bar A\big);}
\]
and
\item for every morphism $g:C'\to C$ in $\cat {comon}$, there is a natural isomorphism of bundles
\[
\xymatrix{g^{*}\big(\xi (C)\big) \ar[d]^\cong &=&\big(\Om C\ar[d]^=\ar [r]&C' \square_{C} \P C\ar[d]_{}^{\cong}\ar[r]&C'\big)\ar[d]^=\\
(\Om g)_{*}\big(\xi (C')\big)&=&\big(\Om C\ar [r] &\P C'\ot_{\Om C'}\Om C\ar [r] &C'\big)}
\]
\end{enumerate}
\end{defn}

\begin{rmk}  Note that axiom (1) of the definition above implies in particular that there are isomorphisms of bundles
\[
\xymatrix{u_{C}^{*}\big(\zeta (\Om C)\big) \ar[d]^\cong &=&\big(\Om C\ar[d]^=\ar [r]&C \square_{\Bar \Om C} \U \Om C\ar[d]_{}^{\cong}\ar[r]&C\big)\ar[d]^=\\
\xi (C)&=&\big(\Om C\ar [r] &\P C\ar [r] &C\big)}
\]
and
\[
\xymatrix{(v_{A})_{*}\big(\xi (\Bar A)\big) \ar[d]^\cong &=&\big(A\ar[d]^=\ar [r]&\P \Bar A\ot_{\Om \Bar A} A\ar[d]_{}^{\cong}\ar[r]&\Bar A\big)\ar[d]^=\\
\zeta (A)&=&\big(A\ar [r] &\U A\ar [r] &\Bar A\big).}
\]
\end{rmk}

Our interest in twisting structures is motivated by the existence of the constructions described in the next definition, of which part (1) is modeled on the Borel construction associated to the action of a topological group on a topological space, while part (2) is analogous to the definition of the homotopy fiber of a continuous map.

\begin{defn}  Let $(\cat M, \otimes, I)$ be monoidal category endowed with a twistable triple $(\cat {mon}, \cat {comon}, \operatorname{mix})$ and a twisting structure  $(\Om, \Bar, \zeta, \xi)$.
\begin{enumerate}
\item Let $f:A\to A'$ be a morphism of augmented monoids in $\cat N$.  The \emph{Borel quotient} of $f$ is
\[
A'\hquot A=\U A\otimes _{A} A'\in{}_{\Bar A}\cat {Mix}_{A'},
\]
where $A'$ is considered as a left $A$-module via the structure induced by $f$.
\item Let $g:C'\to C$ be a morphism of coaugmented comonoids in $\cat N$.  The \emph{Borel kernel} of $g$ is
\[
C\hker C'=C' \square _{C} \P C\in{}_{C'}\cat {Mix}_{\Om C},
\]
where $C'$ is considered as a right $C$-comodule via the structure induced by $g$.
\end{enumerate}
\end{defn}

The Borel quotient and Borel kernel constructions fit into particularly interesting biprincipal bundles.

\begin{defn}\label{defn:puppe} Let $(\cat M, \otimes, I)$ be a monoidal category endowed with a twistable triple $(\cat {mon}, \cat {comon}, \operatorname{mix})$ and a twisting structure  $(\Om, \Bar, \zeta, \xi)$.
\begin{enumerate}
\item Let $f:A\to A'$ be a morphism $\cat {mon}$.   The \emph{Nomura-Puppe sequence} associated to $f$ is the sequence of morphisms in $\cat M$
\[
A\xrightarrow f A' \xrightarrow {\pi_{f}} A'\hquot A \xrightarrow {\delta _{f}} \Bar A \xrightarrow{\Bar f} \Bar A',\]
where $\pi_{f}=j_{A}\otimes_{A}A'$ and $\delta _{f}=\U A\otimes _{A} \ve_{A'}$.
\item Let $g:C'\to C$ be a morphism in $\cat {comon}$.   The \emph{dual Nomura-Puppe sequence} associated to $g$ is the sequence of morphisms in $\cat M$
\[
\Om C'\xrightarrow {\Om g}\Om C \xrightarrow {\del_{g}} C\hker C'\xrightarrow {\iota _{g}} C' \xrightarrow{g} C,
\]
where $\iota_{g}=C'\square_{C}p_{C}$ and $\del _{g}=\eta_{C'}\square _{C} \P C$.
\end{enumerate}
\end{defn}

\begin{rmk} Using the notation of the definition above, observe that
\[
(A' \xrightarrow {\pi_{f}} A'\hquot A \xrightarrow {\delta _{f}} \Bar A)=f_{*}\big(\zeta (A)\big)\cong (\Bar f)^*\big( \zeta(A')\big),
\]
\[
(\Om C \xrightarrow {\del_{g}} C\hker C'\xrightarrow {\iota _{g}} C' )=g^*\big( \xi(C)\big)\cong (\Om g)_{*}\big(\xi(C')\big),
\]
and that both are biprincipal bundles.
\end{rmk}

The bundles defined above are members of a particularly interesting class of biprincipal bundles.

\begin{defn}  Let $(\cat M, \otimes, I)$ be a monoidal category endowed with a twistable triple $(\cat {mon}, \cat {comon}, \operatorname{mix})$ and a twisting structure  $(\Om, \Bar, \zeta, \xi)$.  An  object $\zeta=(A\to N\to C)$ of $\cat {biprin}$ is \emph{classifiable} with respect to the given twisting structure if  there exists a comonoid map $g:C\to \Bar A$ such that
\[
\zeta \cong g^*\big(\zeta(A)\big).
\]
The morphism $g$ is called the \emph{classifying map} of the bundle $\zeta$.
\end{defn}

\begin{rmk} Recall that  axiom (1) of Definition \ref{defn:twist} implies that $\zeta \cong g^*\big(\zeta(A)\big)$ if and only if $\zeta \cong (g^\flat)_{*}\big(\xi(C)\big)$, where $g^\flat: \Om C\to A$ is the transpose of $g:C\to \Bar A$.
\end{rmk}

\subsection{Twisting structures: examples}

\subsubsection{Simplicial sets}

When $(\cat M,\otimes, I)=(\cat {sSet},\times, *)$, we work with the twisting structure $(\G, \overline \W, \zeta , \xi)$.  We refer the reader to, e.g., \cite{may} for the definition of the Kan loop group functor $\G: \cat {sSet}_{0}\to \cat {sGr}$, where $\cat {sSet}_{0}$ denotes the category of reduced simplicial sets and $\cat {sGr}$ the category of simplicial groups. We recall the definition of its right adjoint $\overline\W: \cat {sGr}\to \cat {sSet}_{0}$ in Appendix \ref{sec:twistingfcns}, where we also sketch the theory of twisting functions, which we use below in defining $\zeta$ and $\xi$.

Recall that since the monoidal structure considered on $\cat {sSet}$ is the categorical product, any object $X$ admits a natural comonoid structure given by the diagonal map $\Delta_{X}$, while a (left or right) $(X, \Delta_{X})$-comodule structure on an object $Y$ corresponds to a simplicial map $f:Y\to X$.   Under this identification, the cotensor product of two  $(X, \Delta_{X})$-comodules $f:Y\to X$ and $g:Z\to X$ is exactly their pullback $Y\underset X\times Z$.

For any simplicial group $G$,
\[
\zeta (G)= (G \xrightarrow {j_{G}} \overline \W G \times_{{\nu_{G}}} G \xrightarrow{q_{G}} \overline \W G),
\]
 the well-known ``universal $G$-bundle,'' where $\nu_{G}:G\to \overline\W G$ is couniversal twisting function, $j_{G}$ is the obvious inclusion and $q_{G}$ the obvious projection.
On the other hand, if $X$ is any reduced simplicial set, then
\[
\xi (X)=(\G X\xrightarrow {i_{X}} X\times_{ {\tau_{X}}} \G X \xrightarrow {p_{X}} X),
\]
the ``path fibration on $X$,'' where $\tau_{X}: X\to \G X$ is the universal twisting function, $i_{X}$ is the obvious inclusion, and $p_{X}$ the obvious projection.

It is an easy exercise in using twisting functions to show that this quasitwisting structure is indeed a twisting structure.

If $f:G\to G'$ is a morphism of simplicial groups, then its Borel quotient is
\[
G'\hquot G =\overline \W G\times_{{f\nu_{G}}} G',
\]
which is a model for the homotopy orbits of the action of $G$ on $G'$ induced by $f$.  In particular, $*\hquot G=\ow G$.
If $g:X'\to X$ is a morphism of reduced simplicial sets, then its Borel kernel is
\[
X\hker X'=X' \times_{\tau_{X} g} \G X,
\]
which is a model for the homotopy fiber of $g$. In particular, $X\hker * =\G X$.

\subsubsection{Chain complexes} To define a twisting structure on $\cat{dgProj}_{\Bbbk}$, we use the machinery recalled in Appendix \ref{sec:twistingcochains}. Let  $\cat {mon}$ and $\cat{comon}$ denote, as before (Example \ref{ex:tt-ch}),  the categories of augmented, connected algebras and of coaugmented, $1$-connected coalgebras in $\cat{dgProj}_{\Bbbk}$.  Let ${}_{C}\cat {mix}_{A}$ also be defined as in Example \ref{ex:tt-ch}.

The twisting structure we use in this case is $(\Om, \Bar, \zeta, \xi)$, where
\[
\Om: \cat{comon}\to \cat {mon}\quad\text{and}\quad \Bar :\cat{comon}\to \cat {mon}
\]
 are the cobar and bar constructions recalled in Appendix \ref{sec:twistingcochains}.  Moreover, for all $A\in \ob \cat{mon}$,
\[
\zeta(A)=(A \xrightarrow {j_{A}} \Bar A  \ot_{{t_{\Bar}} } A \xrightarrow {q_{A}} \Bar A),
\]
where $t_{\Bar}$ is the couniversal twisting cochain, $j_{A}(a)= 1\ot a$ for all $a\in A$ and $q_{A}(w\otimes a)=w \cdot \ve(a)$ for all $w\in \Bar A$ and $a\in A$.  Observe that $\Bar A  \ot_{{t_{\Bar}} } A\in {}_{\Bar A}\cat{mix}_{A}$ as required. For all $C\in \ob\cat {comon}$,
\[
\xi(C)=(\Om C \xrightarrow{i_{C}} C\otimes_{{t_{\Om}} }\Om C \xrightarrow {p_{C}} C),
\]
where $t_{\Om}$ is the universal twisting cochain, $i_{C}(w)=1\otimes w$ for all $w\in \Om C$ and $p_{C}(c\otimes w)= c\cdot \ve (w)$ for all $c\in C$ and $w\in \Om C$.  Again, $C\otimes_{{t_{\Om}} }\Om C\in {}_{C}\cat {mix}_{\Om C}$, as required.

Analogously to the simplicial case, it is an easy exercise in using twisting cochains to show that this quasitwisting structure is indeed a twisting structure.

\bibliographystyle{amsplain}

\bibliography{normal}

\end{document}